\documentclass[10pt,a4paper]{article}
\usepackage{amssymb}
\usepackage{amsmath}
\usepackage{graphicx}
\usepackage[T1]{fontenc} 
\usepackage[english]{babel}
\usepackage{amsthm}
\usepackage{tabulary}
\usepackage{booktabs}
\usepackage{mathtools}
\usepackage{tikz}
\usepackage{comment}
\usepackage{tikz-cd}
\usepackage{shuffle}
\usepackage{comment}
\usepackage{stmaryrd}

\usepackage[toc,page]{appendix}

\usepackage{difftrees}
\tikzset{
  ncone/.pic={
	\draw (0,0)--(0,0.2);
  }
}

\tikzset{
  nctwo/.pic={
    \draw (0,0)--(0,0.2);
	\draw (0.1,0)--(0.1,0.2);
  }
}

\tikzset{
  nctwoW/.pic={
    \draw (0,0.2)--(0,0)--(0.1,0)--(0.1,0.2);
  }
}

\tikzset{
  nctwoWW/.pic={
    \draw (0,0.2)--(0,0)--(0.2,0)--(0.2,0.2);
  }
}

\tikzset{
  ncthreeWW/.pic={
    \draw (0,0.2)--(0,0)--(0.3,0)--(0.3,0.2);
	\draw (0.2,0)--(0.2,0.2);
  }
}

\tikzset{
  ncthree/.pic={
    \draw (0,0)--(0,0.2);
	\draw (0.1,0)--(0.1,0.2);
	\draw (0.2,0)--(0.2,0.2);
  }
}

\tikzset{
  ncthreeW/.pic={
    \draw (0,0.2)--(0,0)--(0.2,0)--(0.2,0.2);
	\draw (0.1,0)--(0.1,0.2);
  }
}

\tikzset{
  ncfour/.pic={
    \draw (0,0)--(0,0.2);
	\draw (0.1,0)--(0.1,0.2);
	\draw (0.2,0)--(0.2,0.2);
	\draw (0.3,0)--(0.3,0.2);
  }
}
\tikzset{
  ncfive/.pic={
    \draw (0,0)--(0,0.2);
	\draw (0.1,0)--(0.1,0.2);
	\draw (0.2,0)--(0.2,0.2);
	\draw (0.3,0)--(0.3,0.2);
	\draw (0.4,0)--(0.4,0.2);
  }
}

\tikzset{
  ncfourW/.pic={
    \draw (0,0.2)--(0,0)--(0.3,0)--(0.3,0.2);
	\draw (0.1,0)--(0.1,0.2);
	\draw (0.2,0)--(0.2,0.2);
  }
}
\tikzset{
  ncfiveW/.pic={
    \draw (0,0.2)--(0,0)--(0.4,0)--(0.4,0.2);
	\draw (0.1,0)--(0.1,0.2);
	\draw (0.2,0)--(0.2,0.2);
	\draw (0.3,0)--(0.3,0.2);
  }
}

\tikzset{
  nconeinsidetwoWW/.pic={
    \path (0,0) pic {nctwoWW}; \path (0.1,0.1) pic {ncone};
  }
}

\tikzset{
  nconeinsidethreeWW/.pic={
    \path (0,0) pic {ncthreeWW}; \path (0.1,0.1) pic {ncone};
  }
}
\tikzset{
  nconeinsidethreerightWW/.pic={
    \draw (0,0.2)--(0,0)--(0.3,0)--(0.3,0.2);
    \draw (0.1,0)--(0.1,0.2); \draw (0.2,0.1)--(0.2,0.3);
  }
}
\tikzset{
  nconeinsidethreeleftWW/.pic={
    \draw (0,0.2)--(0,0)--(0.3,0)--(0.3,0.2);
    \draw (0.2,0)--(0.2,0.2); \draw (0.1,0.1)--(0.1,0.3);
  }
}
\tikzset{
  nctwoWWW/.pic={
    \draw (0,0.2)--(0,0)--(0.3,0)--(0.3,0.2);
  }
}

\tikzset{
  nconeoneinsidetwoWW/.pic={
	\path (0,0) pic {ncone};
    \path (0.1,0) pic {nctwoWW}; 
	\path (0.2,0.1) pic {ncone};
  }
}

\usepackage{forest}
\forestset{
	decor/.style = {
		label/.expanded = {[inner sep = 0.2ex, font=\unexpanded{\tiny}]right:{$#1$}}
	},
	root/.style = {minimum size = 0.1ex},
	decorated/.style = {
		for tree = {
			circle, fill, inner sep = 0.3ex, minimum size = 1.ex,
			grow' = north, l = 0, l sep = 1.2ex, s sep = 0.7em,
			fit = tight, parent anchor = center, child anchor = center,
			delay = {decor/.option = content, content =}
		}
	},
	default preamble = {decorated, root},
	begin draw/.code={\begin{tikzpicture}[baseline={([yshift=-0.5ex]current bounding box.center)}]},
}

\DeclareMathOperator{\dgraft}{\widehat{\triangleright}}
\DeclareMathOperator{\dgraftt}{\widehat{\blacktriangleright}}
\newtheorem{theorem}{Theorem}[section]
\newtheorem{corollary}[theorem]{Corollary}
\newtheorem{lemma}[theorem]{Lemma}
\newtheorem{proposition}[theorem]{Proposition}
\newtheorem{definition}[theorem]{Definition}
\newtheorem{example}[theorem]{Example}
\newtheorem{remark}[theorem]{Remark}

\DeclareMathOperator{\graft}{\triangleright}

\title{A Survey on the Munthe-Kaas--Wright Hopf Algebra}

\author{Kurusch Ebrahimi-Fard$^*$, Ludwig Rahm\footnote{Department of Mathematical Sciences, Norwegian University of Science and Technology (NTNU), 7491 Trondheim, Norway. \texttt{kurusch.ebrahimi-fard@ntnu.no}, \texttt{ludwig.rahm@ntnu.no}.}}

\begin{document}

\maketitle

%%%%%%%%%%%%%%%%%%%%%%%%%%%%%%%%%%%
%%%%%%%%%%%%%%%%%%%%%%%%%%%%%%%%%%%

\begin{abstract}
We survey the Munthe-Kaas--Wright Hopf algebra defined on planar rooted trees. This algebra serves a role akin to that of the Butcher--Connes--Kreimer Hopf algebra on non-planar rooted trees within the domain of numerical methods for ordinary differential equations. In the course of our presentation, we revisit Foissy's work on finite-dimensional comodules over the Butcher–Connes–Kreimer Hopf algebra and expand on his findings to include the Munthe-Kaas–Wright Hopf algebra. This involves detailing its endomorphisms and recursively constructing its primitive elements. These results are applied within the context of rough paths, where we describe an isomorphism between planarly branched and geometric rough paths. Our approach hinges on the extension of the Guin--Oudom construction to post-Lie algebras. In the case of the free post-Lie algebra defined on planar rooted trees, it yields the dual of the Munthe-Kaas--Wright Hopf algebra. Surprisingly, we uncover a natural connection between the concept of bialgebras in cointeraction and the Guin--Oudom construction. Prompted by the rough path perspective, we explore this finding through the lens of translations on rough paths. Additionally, we investigate the geometric embedding for planar regularity structures.
\end{abstract}

%%%%%%%%%%%%%%%%%%%%%%%%%%%%%%%%%%%
%%%%%%%%%%%%%%%%%%%%%%%%%%%%%%%%%%%

\tableofcontents

%%%%%%%%%%%%%%%%%%%%%%%%%%%%%%%%%%%
%%%%%%%%%%%%%%%%%%%%%%%%%%%%%%%%%%%

\section{Introduction}
\label{sec:intro}

The scope of classical numerical integration techniques for ordinary differential equations has expanded from its traditional domain in Euclidean spaces to encompass manifold applications. See, for instance, the early work by Crouch and Grossman \cite{Crouch1993aa}. Notably, advancements in Lie group methods \cite{IserlesEtAl2000} have led to the generalization of Butcher’s B-series, giving rise to the concept of Lie--Butcher series \cite{MuntheKaas1995aa}. 

The pioneering works of Connes and Kreimer \cite{CK1998} and of Brouder \cite{Brouder2000aa} revealed that a comprehensive mathematical understanding of the theory of B-series naturally involves the theory of combinatorial Hopf and pre-Lie algebras defined on non-planar rooted trees. Indeed, the seminal works just mentioned highlight the pivotal role played by the Butcher--Connes--Kreimer (BCK) Hopf algebra in providing a natural algebraic and combinatorial framework for describing the composition of B-series \cite{Chartier2010aa,HLW2006,Murua2006aa}, thereby providing powerful ways to study order conditions of numerical methods. Work by Guin and Oudom \cite{OudomGuin2008} further underscores the intricate interplay between properties of rooted trees and the algebraic framework for B-series by showing how the BCK Hopf algebra can be derived from the free pre-Lie algebra. The latter is constructed through a simple and natural grafting operation defined on the space spanned by non-planar rooted trees \cite{ChapLiv2001}. 

Conversely, the theory of Lie--Butcher series necessitates a shift from non-planar rooted trees to planar counterparts. This transition involves replacing the free pre-Lie algebra on non-planar trees with the free post-Lie algebra defined on planar rooted trees \cite{Munthe-Kaas2013aa}. The extension of the Guin--Oudom findings to post-Lie algebras \cite{EbrahimiFardLundervoldMuntheKaas2015} enables the derivation of the Munthe-Kaas--Wright (MKW) Hopf algebra \cite{Munthe-KaasWright2008}. This result sets the stage for our comprehensive examination of the latter. We aim to put the MKW Hopf algebra on the same level as the BCK Hopf algebra, by extending the known results on the latter onto the former.
 
Our presentation centers on the rough path perspective on Lie--Butcher series and the pivotal role played by the MKW Hopf algebra in the context of planarly branched rough paths. In this respect, we delve into Foissy's analysis of finite-dimensional comodules over the BCK Hopf algebra \cite{Foissy2002}. Building upon this work, we broaden its scope\footnote{We also mention earlier work by Foissy \cite{FoissyFrench}, where it is shown that under certain assumptions a coalgebra is isomorphic to the tensor coalgebra. See Remark \ref{rmk:foissyfrench}. We thank one of the referees for bringing this work to our attention.} by including the MKW Hopf algebra. This entails studying its endomorphisms as well as the recursive construction of its primitive elements, which is essential for deeper insights into its structure and properties. Indeed, with Foissy's results at hand, we present an isomorphism between planarly branched and geometric rough paths. With the aforementioned findings in place, and motivated by the recent work of Bruned and Katsetsiadis on the geometric embedding for regularity structures \cite{BrunedKatsetsiadis2023}, we investigate the geometric embedding for planar regularity structures. 

The compositional structure of B-series and Lie--Butcher series is complemented by substitution laws, which are crucial in the context of modified differential equations within backward error analysis \cite{CHV2007}. Recognizing the significance of the concept of cointeracting bialgebras in the context of composition and substitution laws for both B-series and Lie--Butcher series \cite{CEFM2011,Chartier2010aa,Manchon2018,Rahm2022a}, we unveil a natural yet previously overlooked linkage between bialgebras in cointeraction and the Guin--Oudom framework. Motivated by the rough path perspective, we further explore this connection through the lens of translations on rough paths presented in reference \cite{BrunedChevyrevFrizPreiss,Rahm2021RP}. We remark that Foissy recently formalised the concept of bialgebras in cointeracting into the notion of double bialgebras \cite{Foissy2022}.

\medskip

Let us start with the common initial value problem for a smooth vector field $f$
\begin{align}
\label{eq:ivp}
	y'=f(y), \quad y(0)=y_0,
\end{align}
and consider the numerical method
$$
	y_{k+1}=\Phi(h,f)(y_k) .
$$ 
It is well-known that in the case of a Runge--Kutta method the Taylor expansion of $\Phi(h,f)(y_0)$, seen as a function of the step-size parameter $h$, can be expressed as a linear combination of so-called elementary differentials in the vector field $f$. Back in 1857,  Cayley \cite{Cayley1857} 
%Cayley1881
understood that elementary differentials are in bijection with non-planar rooted trees. This fundamental observation plays an important role in the work of Butcher \cite{Butcher1972} who developed an algebraic theory of integration methods based on non-planar rooted trees
\begin{equation*}
	\textbf{T} := \{\Forest{[]},\Forest{[[]]},\Forest{[[[]]]},\Forest{[[][]]},\Forest{[[[[]]]]},
	\Forest{[[[][]]]},\Forest{[[[]][]]}=\Forest{[[][[]]]},\Forest{[[][][]]} ,\ldots   \}.
\end{equation*}
The degree of a rooted tree $\tau \in T$ is defined in terms of its number of vertices denoted $|\tau|$. The empty tree $e$ can be seen as having degree zero. We can order the set $\textbf{T}  = \bigcup_{n \ge 0} \textbf{T} _n$ in terms of degrees. Looking at a rooted tree $\tau$, it becomes evident that it can be described in terms of a set of smaller rooted trees $\tau_1,\ldots,\tau_n$ and the $B_+$-operator. The latter adds a new root and connects it via new edges to the roots of the $\tau_i$, such that $\tau=B_+(\tau_1,\ldots,\tau_n)$ and $|\tau|=1+|\tau_1| + \cdots +|\tau_n|$. For instance, $B_+(e)=\Forest{[]}$ and $|\Forest{[]}|=1 + 0$.  

More generally, we say that the numerical method $\Phi(h,f)$ is a Butcher or simply B-series method if its Taylor expansion is a linear combination of elementary differentials. In this case, there exists a function $\alpha: \textbf{T} \to \mathbb{R}$ such that
\begin{equation} 
\label{eq::BSeriesMethod}
	\Phi(h,f)(y_k)= \sum_{\tau \in \textbf{T} } h^{|\tau|}\alpha(\tau)\mathcal{F}_f[\tau](y_k),
\end{equation}
where $\mathcal{F}_f$ is a bijection between rooted trees and elementary differentials defined inductively \cite{HLW2006} by $\mathcal{F}_f[B_+(e)]:=f$ and  
\begin{equation} 
\label{eq:elementarydiff} 
	\mathcal{F}_f[B_+(\tau_1,\ldots,\tau_n)]:= f^{(n)}(F_f[\tau_1],\ldots, F_f[\tau_n]).
\end{equation}
Here the $n$-th derivative of $f$ is denoted $f^{(n)}$. For example 
\begin{equation*} 
	\mathcal{F}_f[\Forest{[[]]}]=f^{(1)}(f)
	\qquad
	\mathcal{F}_f[\Forest{[[][]]}]=f^{(2)}(f,f)
	\qquad
	\mathcal{F}_f[\Forest{[[[]]]}]=f^{(1)}(f^{(1)}(f)).
\end{equation*}

There is a bijection between B-series methods and maps from $\textbf{T}$ to $\mathbb{R}$. We will denote the B-series method given by the map $\alpha$ as $B_f(\alpha)$. Runge--Kutta methods are dense in the space of B-series methods \cite{butcher2016numerical}, but not every B-series method is a Runge--Kutta method. Indeed, there are known examples of B-series methods that are not Runge--Kutta. We note that the natural question of which numerical integration methods can be described by B-series was answered recently in the important work \cite{MclachlanModinMuntheKaasVerdier2016}. 

Vector fields on a Euclidean space, together with the directional derivative, form a so-called pre-Lie algebra \cite{Burde2006,CP2021,Manchon2011}. Denoting the derivative of the vector field $Y$ in the direction of the vector field $X$ by $\triangledown_X Y$, this means that they satisfy the algebraic relation
\[
	\triangledown_X(\triangledown_Y Z) - \triangledown_{ \triangledown_X Y}Z 
	= \triangledown_Y(\triangledown_X Z) - \triangledown_{\triangledown_Y  X}Z,
\]
where we interpret $\triangledown$ as being a product in the algebra on vector fields. The free pre-Lie algebra is given by the vector space $\mathcal{T}=\text{span}(\textbf{T})$ spanned by non-planar rooted tree, together with the grafting product $\curvearrowright$. The product $\tau_1 \curvearrowright \tau_2$ is defined as the sum over all ways to add an edge from the root of $\tau_1$ to some vertex of $\tau_2$, for example:
\[
	\Forest{[[]]}\curvearrowright \Forest{[[][]]}=\Forest{[[[]][][]]}+\Forest{[[[[]]][]]}+\Forest{[[][[[]]]]}.
\]
The free pre-Lie algebra has a universality property \cite{ChapLiv2001}. In the above context, it says that there exists a unique pre-Lie algebra morphism from rooted trees to the vector fields, generated by sending the single-vertex tree to a vector field $f$. This unique morphism is indeed the map $\mathcal{F}_f$ from Equation \eqref{eq::BSeriesMethod} respectively Equation \eqref{eq:elementarydiff}.

In \cite{OudomGuin2008}, Guin and Oudom constructed the so-called Guin--Oudom functor which maps a pre-Lie algebra to a Hopf algebra. Following this functor one arrives at the important BCK Hopf algebra $\mathcal{H}_{\scriptscriptstyle{\textrm{BCK}}}$ over the non-planar forests $\mathcal{F}=\text{span}(\textbf{F})$, given by commutative words in non-planar trees \cite{CK1998}. The coproduct $\Delta_{\scriptscriptstyle{\textrm{BCK}}}$ of $\mathcal{H}_{\scriptscriptstyle{\textrm{BCK}}}$ describes the (dual of the) operation of composing two B-series methods. Indeed, first taking one step with the method $B_f(\alpha)$ and then taking one step with another method, $B_f(\beta)$, can be described as taking one step with the method $B_f(\gamma)$, where the map $\gamma$ on a forest $\omega \in \mathcal{F}$ is computed by
\begin{align*}
	\langle \gamma,\omega \rangle = \langle \alpha \otimes \beta, \Delta_{\scriptscriptstyle{\textrm{BCK}}}(\omega) \rangle.
\end{align*}
It is important to note that the Guin--Oudom functor has become a crucial tool in constructing combinatorial Hopf algebras used to describe rough paths and regularity structures \cite{BruKatPostLie,BrunedManchon2022}.

Recalling that B-series can describe vector fields, one can also consider expressions of the form $B_{B_f(\alpha)}(\beta)$. This is called substitution of B-series and can be described with a similar method, i.e., $B_{B_f(\alpha)}(\beta)=B_f(\upsilon)$, where $\upsilon$ is obtained by a second Hopf algebra $\mathcal{H}_{\scriptscriptstyle{\textrm{CEFM}}}$ that is coacting on $\mathcal{H}_{\scriptscriptstyle{\textrm{BCK}}}$ via a coaction $\rho_{\scriptscriptstyle{\textrm{CEFM}}}: \mathcal{H}_{\scriptscriptstyle{\textrm{BCK}}} \to \mathcal{H}_{\scriptscriptstyle{\textrm{CEFM}}} \otimes \mathcal{H}_{\scriptscriptstyle{\textrm{BCK}}}$ \cite{CEFM2011}. Crucially, this coaction satisfies an important property called \textit{cointeraction}. We emphasise that substitution of B-series has been linked to translations of rough paths and renormalisation of regularity structures \cite{BrunedChevyrevFrizPreiss}. 

\smallskip
 
Consider again the IVP \eqref{eq:ivp}, where now $y: \mathbb{R} \to \mathcal{M}$ is valued in a homogeneous space $\mathcal{M}$ and $f$ is a vector field on $\mathcal{M}$. Runge--Kutta methods for this equation were developed by Munthe-Kaas in \cite{MuntheKaas1995aa,MuntheKaas1998}. Here one makes the same observation as in Euclidean space, i.e., the Taylor expansion of Munthe-Kaas--Runge--Kutta methods are linear combinations of elementary differentials. The so-called Lie--Butcher series (LB-series) methods are then defined as any method whose Taylor expansion is a linear combination of elementary differentials. It was shown in \cite{Munthe-KaasWright2008} that elementary differentials can be expanded as a series in (the free Lie algebra generated by) \textit{planar} rooted trees
\begin{align*}
	\textbf{PT} = \{\Forest{[]},\Forest{[[]]},\Forest{[[[]]]},\Forest{[[][]]},\Forest{[[[[]]]]},
	\Forest{[[[][]]]},\Forest{[[[]][]]} \neq \Forest{[[][[]]]},\Forest{[[][][]]} ,\dots   \},
\end{align*}
meaning rooted trees endowed with an embedding into the plane. The order of the branches attached to a vertex now matters, which reflects the non-commutativity of covariant derivatives on a manifold. In order to describe this representation, consider the space $\mathcal{X}(\mathcal{M})$ of vector fields on a manifold $\mathcal{M}$, together with a covariant derivative where we denote $X \graft Y$ for the derivative of the vector field $Y$ in the direction of $X$. Recall that the torsion $T$ is defined by
\[
	T(X,Y)=X \graft Y - Y \graft X - \llbracket X,Y \rrbracket,
\]
where $\llbracket \cdot,\cdot \rrbracket$ is the Jacobi bracket. If the connection is flat and has constant torsion, meaning that we are in a homogeneous space, then $(\mathcal{X}(\mathcal{M}),-T(\cdot,\cdot),\graft)$ is a so-called \textit{post-Lie algebra}. The free post-Lie algebra is given by the free Lie algebra generated over planar trees, and the representation of vector fields by planar trees in LB-series is the unique post-Lie algebra morphism given by the universality property of the free post-Lie algebra. Note that pre-Lie algebra is a special case of post-Lie algebra. In the geometric setting outlined above it corresponds to the fact that Euclidean spaces are a special case of homogeneous spaces. 

The Guin--Oudom functor was extended to post-Lie algebras in \cite{EbrahimiFardLundervoldMuntheKaas2015}. Applying this functor in the context of the free post-Lie algebra, one arrives at the MKW Hopf algebra $\mathcal{H}_{\scriptscriptstyle{\textrm{MKW}}}$ -- used to describe composition of LB-series \cite{Munthe-KaasWright2008}. This may be interpreted as a planar generalization of the BCK Hopf algebra. Indeed, $\mathcal{H}_{\scriptscriptstyle{\textrm{BCK}}}$ is a Hopf subalgebra of $\mathcal{H}_{\scriptscriptstyle{\textrm{MKW}}}$. Indeed, it was shown in \cite{Munthe-KaasWright2008} that there exists an injective, non-surjective, Hopf algebra morphism $\Omega: \mathcal{H}_{\scriptscriptstyle{\textrm{BCK}}} \to \mathcal{H}_{\scriptscriptstyle{\textrm{MKW}}}$ defined by sending a non-planar forest to a weighted sum of all its possible planar embeddings. The fact that $\mathcal{H}_{\scriptscriptstyle{\textrm{BCK}}}$ embeds as a subspace in $\mathcal{H}_{\scriptscriptstyle{\textrm{MKW}}}$ makes sense, corresponding to the fact that every Euclidean space is also homogeneous. Continuing the analogy between the Hopf algebras, it is also the case that the MKW coproduct $\Delta_{\scriptscriptstyle{\textrm{MKW}}}$ describes composition of LB-series in exactly the same way that $\Delta_{\scriptscriptstyle{\textrm{BCK}}} $ describes composition of B-series. Substitution of LB-series, on the other hand, is more complicated and was first described in \cite{Lundervold2013aa}. It was explored further by one of us in \cite{Rahm2022a}, where its proper Hopf algebraic formulation was given fitting the framework of cointeracting bialgebras.

In the important work \cite{Foissy2002}, Foissy characterised all finite-dimensional comodules as well as all endomorphisms of the BCK Hopf algebra $\mathcal{H}_{\scriptscriptstyle{\textrm{BCK}}}$. He furthermore constructed a recursively defined projection map onto the primitive elements of $\mathcal{H}_{\scriptscriptstyle{\text{BCK}}}$. To prove these results he considered the so-called natural growth operation 
$$
	\top: \mathcal{H}_{\scriptscriptstyle{\text{BCK}}} \otimes \mathcal{H}_{\scriptscriptstyle{\mathrm{BCK}}} 
	\to \mathcal{H}_{\scriptscriptstyle{\text{BCK}}}
$$ 
and its relation with the reduced coproduct $\hat{\Delta}_{\scriptscriptstyle{\text{BCK}}}$ on $\mathcal{H}_{\scriptscriptstyle{\text{BCK}}}$:
\begin{equation}
\label{eq::NaturalGrowthBCK}
	\hat{\Delta}_{\scriptscriptstyle{\text{BCK}}}(x \top y) 
	= x \otimes y + x^{(1)} \otimes (x^{(2)} \top y ),
\end{equation}
for any $x \in \mathcal{H}_{\scriptscriptstyle{\text{BCK}}}$ and any primitive element $y$, i.e., $\hat{\Delta}_{\scriptscriptstyle{\text{BCK}}}(y)=0$. On the righthand side of \eqref{eq::NaturalGrowthBCK}, we used Sweedler's notation for the reduced coproduct, i.e., $\hat{\Delta}_{\scriptscriptstyle{\text{BCK}}}(x)=x^{(1)} \otimes x^{(2)}$. In the recent works \cite{VarzanehRiedelSchmedingTapia2022, bellingeri2024branched}, Foissy's description of primitive elements in $\mathcal{H}_{\scriptscriptstyle{\text{BCK}}}$ in terms of the natural growth operator was revisited in the context of Gubinelli's controlled rough paths.

\smallskip

In the paper at hand, we define such a natural growth operator on the MKW Hopf algebra $\mathcal{H}_{\scriptscriptstyle{\text{MKW}}}$ of planar forests, and use Foissy's results to describe its primitive elements. With the goal of making our presentation self-contained, we recall Foissy's work by extending it to any combinatorial Hopf algebra which permits the definition of a natural growth operator\footnote{See, however, Footnote 1 above.}. Along the way, links are observed between the concept of bialgebras in cointeraction \cite{CEFM2011,Foissy2022}, the Guin--Oudom construction \cite{OudomGuin2008} and the notion of translation on rough paths \cite{BrunedChevyrevFrizPreiss,Rahm2021RP}. We then apply these findings to describe an isomorphism between planarly branched and geometric rough paths, similar to the work \cite{BoedihardjoChevyrev2019}. Following the recent work by the second author on Hairer's regularity structures \cite{Hairer2014regularity} from the viewpoint of planarly branched rough paths \cite{CurryEbrahimiFardManchonMuntheKaas2018}, we explore Bruned's and Katsetsiadis' geometric embedding \cite{BrunedKatsetsiadis2023} in the context of planar regularity structures \cite{Rahm2022}. 

\medskip

The paper is organised as follows. We will first recall the Guin--Oudom construction in Section \ref{sec:GOconstruction}, which underlies both the BCK Hopf algebra and the MKW Hopf algebra. We then show in Theorem \ref{thm::GraftCointeraction} how this construction induces a cointeraction between two isomorphic Hopf algebras. In Section \ref{sec:MKW}, we recall the MKW Hopf algebra obtained by applying the Guin--Oudom construction to the free post-Lie algebra. In Section \ref{sect:Foissy}, we recall the main results of Foissy's paper \cite{Foissy2002} on the BCK Hopf algebra, extended to any Hopf algebra with a natural growth operator satisfying the analog of equation \eqref{eq::NaturalGrowthBCK} with respect to its reduced coproduct. Unless explicitly stated otherwise, all results in this section are due to Foissy \cite{Foissy2002} (see also Remark \ref{rmk:foissyfrench}). We then show that the MKW Hopf algebra has a natural growth operator, and apply Foissy's results on this special case. Using Proposition \ref{prop::Projection}, we describe a projection onto the primitive elements in the MKW Hopf algebra. In Proposition \ref{Prop:HopfAlgebraMorphism} we show that the MKW Hopf algebra is isomorphic to a shuffle Hopf algebra. In Section \ref{sec:roughpaths}, we recall the concepts of rough paths in combinatorial Hopf algebras and their translations. We then comment on the cointeraction obtained from the Guin--Oudom construction in Section \ref{sec:GOconstruction}, and its relations to the notion of translations of rough paths. We show in Proposition \ref{prop::DisjointCointeractions} that the two different cointeractions are disjoint. In Section \ref{sec:planarlyBRP}, we show an isomorphism between planarly branched rough paths and geometric rough paths. We then argue in Corollary \ref{Cor::TrivialSignatureTreeLike} that a planarly branched rough path has trivial signature if and only if it is tree-like. In Section \ref{sec:geomembedding}, we explore the geometric embedding for regularity structures \cite{BrunedKatsetsiadis2023} in the context of planar regularity structures. The geometric embedding in \cite{BrunedKatsetsiadis2023} comes in the form of a Hopf algebra isomorphism between the Hopf algebra for recentering and the quotient of a tensor Hopf algebra. We comment on why one needs to take a quotient, and construct the corresponding isomorphism in the planar setting.

\medskip
\medskip

\noindent {\bf{Acknowledgements}}: We thank D.~Manchon for helpful comments and remarks. We would also like to thank the referees for helpful comments which led to an improved presentation. We also thanks Darij Grinberg for helpful comment. Both authors are supported by the Research Council of Norway through project 302831 “Computational Dynamics and Stochastics on Manifolds” (CODYSMA). The first author would like to thank the Center for Advanced Study in Oslo for warm hospitality and support. The second author has received support from the project Pure Mathematics in Norway – Ren matematikk i Norge funded by the Trond Mohn Foundation.

%%%%%%%%%%%%%%%%%%%%%%%%%%%%%%%%%%%
%%%%%%%%%%%%%%%%%%%%%%%%%%%%%%%%%%%

\section{Post-Lie algebras and the Guin--Oudom construction}
\label{sec:GOconstruction}

In reference \cite{OudomGuin2008}, Guin and Oudom describe what may be called the Guin--Oudom construction. It provides a method to obtain a Hopf algebra starting from a pre-Lie algebra. See \cite{Burde2006,Manchon2008} for details regarding pre-Lie algebras. The Guin--Oudom construction can be generalised to post-Lie algebras \cite{EbrahimiFardLundervoldMuntheKaas2015}. See also \cite{Foissy2018post}.

Following \cite{Munthe-KaasWright2008}, a post-Lie algebra $(V,\graft,[\cdot,\cdot])$ consists of a vector space $V$ endowed with a Lie bracket $[\cdot,\cdot]$ as well as a magmatic (post-Lie) product $\graft: V \otimes V \to V$ such that the following relations are satisfied for all $x,y,z \in V$:
\begin{align}
	[x,y]\graft z
	&= x \graft (y \graft z) - (x \graft y)\graft z - y \graft (x \graft z) + (y \graft x)\graft z,\label{postLie1}\\ 
	x \graft [y,z]
	&= [x \graft y,z] + [y,x \graft z]. \label{postLie2}
\end{align}
One can show that a post-Lie algebra always comes equipped with a second Lie bracket defined by:
\begin{align*}
	\llbracket x,y \rrbracket := x \graft y - y \graft x + [x,y].
\end{align*}
It is clear from \eqref{postLie1} and \eqref{postLie2} that a post-Lie algebra is a generalisation of a (left) pre-Lie  algebra in the sense that the latter can be seen as a post-Lie algebra with a trivial Lie bracket, $[x,y]=0$ for all $x,y \in V$.

Let $\mathfrak{g}=(V,[\cdot,\cdot])$ and $\tilde{\mathfrak{g}}=(V,\llbracket \cdot,\cdot \rrbracket)$ denote the two Lie algebras present in the context of a post-Lie algebra, with corresponding universal enveloping algebras $\mathcal{U}(\mathfrak{g})$ respectively $\mathcal{U}(\tilde{\mathfrak{g}})$. In \cite{EbrahimiFardLundervoldMuntheKaas2015}, the Guin--Oudom construction was generalised to post-Lie algebras, by first extending the post-Lie product $\graft$ to $\mathcal{U}(\mathfrak{g})$:
\begin{align}
	1 \graft A 
	&= A 							\label{eq::GuinOudom1} \\
	xA \graft y 
	&= x \graft (A \graft y) - (x \graft A) \graft y \label{eq:.GuinOudom2} \\
	A \graft BC 
	&= (A_{(1)} \graft B)(A_{(2)} \graft C),	\label{eq::GuinOudom3}
\end{align}
where $x,y \in \mathfrak{g}$ and  $A,B,C \in \mathcal{U}(\mathfrak{g})$. Recall that $\mathcal{U}(\mathfrak{g})$ carries a natural Hopf algebra structure, which we denote by $\mathcal{H}_T=(\mathcal{U}(\mathfrak{g}),\cdot,\Delta_{\shuffle})$. Note that we are using Sweedler's notation for the unshuffle coproduct, $\Delta_{\shuffle}(A)=A_{(1)} \otimes A_{(2)}$, defined on $\mathcal{U}(\mathfrak{g})$. The structure $(\mathcal{U}(\mathfrak{g}),\graft,\cdot)$ is a so-called D-algebra \cite{Munthe-KaasWright2008}. Following Guin--Oudom, we now define a non-commutative product on $\mathcal{U}(\mathfrak{g})$:
\begin{align}
\label{assoprod}
	A \ast B 
	&= A_{(1)}\cdot (A_{(2)} \graft B), 
\end{align}
which can be shown to be associative and unital. In fact, the product \eqref{assoprod} enters the Guin--Oudom construction through the intriguing identity \cite{EbrahimiFardLundervoldMuntheKaas2015,OudomGuin2008}.
\begin{align}
\label{shift}	
	A \graft (B \graft C) = (A \ast B) \graft C. 
\end{align}

\begin{theorem}[\cite{EbrahimiFardLundervoldMuntheKaas2015}] \label{Theorem::GuinOudomIsomorphism}
$\mathcal{H}=(\mathcal{U}(\mathfrak{g}),\ast,\Delta_{\shuffle})$ is a Hopf algebra. Furthermore, $\mathcal{H}$ is isomorphic to the natural Hopf algebra defined on $\mathcal{U}(\tilde{\mathfrak{g}})$.
\end{theorem}

Let $\overline{\mathcal{H}}=(\overline{\mathcal{U}(\mathfrak{g})},\ast,\Delta_{\shuffle})$ denote the completion of $\mathcal{H}$ with respect to the grading, with $\ast$ and $\Delta_{\shuffle}$ extended continuously. Similarly let $\overline{\mathcal{H}}_T$ denote the continuously extended tensor Hopf algebra. Note that both of these extensions are isomorphic as vector spaces, both are algebras, and neither is a coalgebra. That is because $\Delta_{\shuffle} : \overline{\mathcal{U}(\mathfrak{g})} \to \mathcal{U}(\mathfrak{g}) \overline{\otimes} \mathcal{U}(\mathfrak{g})$ does not take values in $\overline{\mathcal{U}(\mathfrak{g})} \otimes \overline{\mathcal{U}(\mathfrak{g})}$. Recall that the group-like elements 
$$
	\mathcal{G}:=\{A \in \overline{\mathcal{U}(\mathfrak{g})}\ |\ \Delta_{\shuffle}(A)=A \otimes A  \}
$$ 
form two groups: $\mathcal{G}_{\ast}=(\mathcal{G},\ast)$ and $\mathcal{G}_{\odot}=(\mathcal{G},\cdot)$. Now consider equation \eqref{eq::GuinOudom3} in the post-Lie version of the Guin--Oudom construction and let $A$ be a group-like element, then:
\begin{align}
\label{eq::GraftIsDistributive}
	A \graft BC=(A\graft B)(A \graft C).
\end{align}
Hence, from \eqref{shift} we deduce that $\graft: \mathcal{G} \otimes \overline{\mathcal{H}}_T \to \overline{\mathcal{H}}_T$ is an action of the group $\mathcal{G}_{\ast}$ on the group $\mathcal{G}_{\odot}$ that distributes over the product in $\mathcal{G}_{\odot}$. This is the dual notion to the important concept of cointeracting bialgebras \cite{CEFM2011,Foissy2022}.

\begin{definition} [\cite{Foissy2022}]
\label{def::cointeraction}
We say that two bialgebras, $(\mathcal{A},\odot_{\scriptscriptstyle{\mathcal{A}}}, \Delta_{\scriptscriptstyle{\mathcal{A}}}\epsilon_{\scriptscriptstyle{\mathcal{A}}},\eta_{\scriptscriptstyle{\mathcal{A}}})$ and  $(\mathcal{B},\odot_{\scriptscriptstyle{\mathcal{B}}},\Delta_{\scriptscriptstyle{\mathcal{B}}},\epsilon_{\scriptscriptstyle{\mathcal{B}}},\eta_{\scriptscriptstyle{\mathcal{B}}})$, are in cointeraction if $\mathcal{B}$ is coacting on $\mathcal{A}$ via a map $\rho: \mathcal{A} \to \mathcal{B} \otimes \mathcal{A}$ that satisfies:
\begin{align*}
	\rho(1_\mathcal{A})
	&=1_\mathcal{B} \otimes 1_\mathcal{A}, \\
	\rho(x \odot_{\scriptscriptstyle{\mathcal{A}}} y)
	&=\rho(x) (\odot_{\scriptscriptstyle{\mathcal{B}}} \otimes \odot_{\scriptscriptstyle{\mathcal{A}}}) \rho(y), \\
	( \mathrm{id} \otimes \epsilon_{\scriptscriptstyle{A}})\rho
	&=1_\mathcal{B} \epsilon_{\scriptscriptstyle{\mathcal{A}}}, \\
	( \mathrm{id} \otimes \Delta_{\scriptscriptstyle{A}})\rho
	&=m_{\scriptscriptstyle{\mathcal{B}}}^{1,3}(\rho \otimes \rho)\Delta_{\scriptscriptstyle{\mathcal{A}}},
\end{align*}
where 
\begin{align*}
	m_{\scriptscriptstyle{\mathcal{B}}}^{1,3}(a \otimes b \otimes c \otimes d)
		:=a \odot_{\scriptscriptstyle{\mathcal{B}}}c \otimes b \otimes d.
\end{align*}
\end{definition}

\begin{remark}
In numerical analysis, the group $\mathcal{G}_{\odot}$ represents numerical methods, whereas the group $\mathcal{G}_{\ast}$ represents flows. The product $\graft$ represents parallel transport \cite{CEFMK2019}. Equality \eqref{eq::GraftIsDistributive} then expresses an equivariance of numerical methods in $\mathcal{G}_{\odot}$: one can either first compose two numerical methods and then parallel transport along a flow, or first parallel transport both methods along the same flow and then compose them.
\end{remark}

Let $\mathcal{U}(\mathfrak{g})^{\ast}$ denote the graded dual of $\mathcal{U}(\mathfrak{g})$, which we can identify with $\mathcal{U}(\mathfrak{g})$ as a vector space by using the dual basis. Let $\mathcal{H}^{\ast}=(\mathcal{U}(\mathfrak{g})^{\ast},\shuffle,\Delta_{\ast})$ and $\mathcal{H}_T^{\ast}=(\mathcal{U}(\mathfrak{g})^{\ast},\shuffle,\Delta_{\odot})$ denote the corresponding graded dual Hopf algebras. Note that the space $\overline{\mathcal{U}(\mathfrak{g})}$ can be identified with linear maps on $\mathcal{U}(\mathfrak{g})^{\ast}$; Group-like elements $\mathcal{G}$ can be identified with Hopf algebra characters. We now define the dual coaction to the product $\graft$, denoted 
$$
	\rho_{\graft}: \mathcal{H}_T^{\ast} \to \mathcal{H}^{\ast} \otimes \mathcal{H}_T^{\ast},
$$ 
in terms of the pairing
\begin{align*}
	\langle A \graft B, x \rangle 
	= \langle A \otimes B, \rho_{\graft}(x) \rangle.
\end{align*}
We will show that $\mathcal{H}^{\ast}$ is in cointeraction with $\mathcal{H}_T^{\ast}$ via the coaction $\rho_{\graft}$.

\begin{lemma}
Let $A,B \in \mathcal{G}$ be group-like elements, then $A \graft B$ is group-like.
\end{lemma}

\begin{proof}
It was shown in \cite{EbrahimiFardLundervoldMuntheKaas2015} that:
\begin{align*}
	\Delta_{\shuffle}(A \graft B)= (A_{(1)} \graft B_{(1)}) \otimes (A_{(2)} \graft B_{(2)}).
\end{align*}
Hence if $A,B$ are group-like elements, then:
\begin{align*}
	\Delta_{\shuffle}(A \graft B)= (A \graft B) \otimes (A \graft B).
\end{align*}
\end{proof}

\begin{lemma}
If $\langle \chi,x\rangle = \langle \chi,y \rangle$ for every character $\chi$, then we have equality $x=y$.
\end{lemma}

\begin{theorem} \label{thm::GraftCointeraction}
The coaction $\rho_{\graft}$ gives a cointeraction of $\mathcal{H}^{\ast}$ onto $\mathcal{H}_T^{\ast}$.
\end{theorem}

\begin{proof}
Let $x,y \in \mathcal{H}_T^{\ast}$ be arbitrary elements and let $A,B,C \in \overline{\mathcal{U}(\mathfrak{g})}$ be characters. We shall check all of the required properties:
\begin{enumerate}
\item $\rho_{\graft}(1)= 1 \otimes 1$ follows from the fact that for every character $A$ we have $\langle A,1\rangle=1$.

\item $\rho_{\graft}(x \shuffle y)=\rho_{\graft}(x) (\shuffle \otimes \shuffle) \rho_{\graft}(y)$ is seen by the computation:
\begin{align*}
	\langle A \otimes B, \rho_{\graft}(x \shuffle y) \rangle 
	&= \langle (A \graft B)\otimes (A \graft B), x \otimes y \rangle \\
	&= \langle A \otimes B \otimes A \otimes B, \rho_{\graft}(x) \otimes \rho_{\graft}(y) \rangle,
\end{align*}
where we now use the group-like properties $\Delta_{\shuffle}(A)=A \otimes A$, $\Delta_{\shuffle}(B)=B \otimes B$ to see that:
\begin{align*}
	\langle A \otimes B, \rho_{\graft}(x \shuffle y) \rangle 
	=\langle A \otimes B,\rho_{\graft}(x)(\shuffle \otimes \shuffle)\rho_{\graft}(y)\rangle.
\end{align*}

\item The property $( \mathrm{id} \otimes \epsilon)\rho_{\graft}=1\epsilon$ follows from $x \graft 1$ being nonzero if and only if $x$ has a component of degree $0$.

\item To see the last property, note that:
\begin{align*}
	\langle A \otimes B \otimes C , ( \mathrm{id} \otimes \Delta_{\odot})\rho_{\graft}(x)\rangle 
	&= \langle A \graft BC, x \rangle \\
	&=\langle (A \graft B)(A \graft C),x \rangle \\
	&=\langle A \graft B \otimes A \graft C,\Delta_{\odot}(x)\rangle\\
	&=\langle A \otimes B \otimes A \otimes C, (\rho_{\graft} \otimes \rho_{\graft})\Delta_{\odot}(x) \rangle.
\end{align*}
Next, we use again $\Delta_{\shuffle}(A)=A \otimes A$ to see that
\begin{align*}
	\langle A \otimes B \otimes C , ( \mathrm{id} \otimes \Delta_{\odot})\rho_{\graft}(x)\rangle 
	= \langle A \otimes B \otimes C, m_{\shuffle}^{1,3}(\rho_{\graft} \otimes \rho_{\graft})\Delta_{\odot}(x)\rangle.
\end{align*}
\end{enumerate}
\end{proof}

As a consequence, we note the following interesting identity.

\begin{proposition}
The coaction $\rho_{\graft}$ satisfies the identities:
\begin{align}
	(\mathrm{id} \otimes \rho_{\graft})\rho_{\graft}
	&=m_{\shuffle}^{1,2}( \mathrm{id} \otimes \rho_{\graft} \otimes \mathrm{id})
	((\Delta_{\odot} \otimes \mathrm{id})\rho_{\graft}  ) \label{eq::CoTranslation}\\
	&=(\Delta_{\ast} \otimes  \mathrm{id} )\rho_{\graft}. \label{eq::CoSubstitution}
\end{align}
\end{proposition}

\begin{proof}
Let $x \in \mathcal{H}_T^{\ast}$ be an arbitrary element and let $A,B,C \in \overline{\mathcal{U}(\mathfrak{g})}$ be characters, we compute:
\begin{align*}
	\langle A \otimes B \otimes C, (\text{id} \otimes \rho_{\graft})\rho_{\graft}(x) \rangle 
	&= \langle A \graft (B \graft C),x \rangle \\
	&= \langle (A(A \graft B))\graft C, x \rangle \\
	&= \langle A(A \graft B) \otimes C , \rho_{\graft}(x) \rangle \\
	&= \langle A \otimes A \graft B \otimes C, (\Delta_{\odot} \otimes \text{id})\rho_{\graft}(x) \rangle\\
	&= \langle A \otimes A \otimes B \otimes C, (\text{id} \otimes 
	\rho_{\graft} \otimes \text{id})( (\Delta_{\odot} \otimes \text{id})	\rho_{\graft}(x)  ) \rangle \\
	&= \langle A \otimes B \otimes C, m_{\shuffle}^{1,2}(\text{id} \otimes \rho_{\graft} 
	\otimes \text{id})( (\Delta_{\odot} \otimes \text{id})\rho_{\graft}(x)  ) \rangle,
\end{align*}
where we used identity \eqref{shift}, i.e., $A \graft (B \graft C) = (A(A\graft B))\graft C$. For the second identity, note that $A(A \graft B)=A \ast B$.
\end{proof}

\begin{remark}
In \cite{Rahm2021RP}, identity \eqref{eq::CoTranslation} was shown to characterise the notion of translation on rough paths. Identifying $\mathcal{H}^{\ast}$ and $\mathcal{H}_T^{\ast}$ as vector spaces lets us see $\rho$ as a map $\mathcal{H}_T^{\ast} \to \mathcal{H}_T^{\ast} \otimes \mathcal{H}_T^{\ast}$, then identity \eqref{eq::CoTranslation} gives a cotranslation of $\mathcal{H}_T^{\ast}$ onto itself. From the viewpoint of duality between translation actions and translation coactions, this identity corresponds to the translation identity $T_{\mathbf{v}} \circ T_{\mathbf{u}} = T_{\mathbf{v}+T_{\mathbf{v}}(u)}$. We elaborate more on this in a later section.
\end{remark}

%%%%%%%%%%%%%%%%%%%%%%%%%%%%%%%%%%%
%%%%%%%%%%%%%%%%%%%%%%%%%%%%%%%%%%%

\section{Munthe-Kaas--Wright Hopf algebra of planar forests}
\label{sec:MKW}

The Munthe-Kaas--Wright (MKW) Hopf algebra was introduced in \cite{Munthe-KaasWright2008}. In a nutshell, it is to numerical integration methods on homogeneous spaces, what the Butcher--Connes--Kreimer Hopf algebra is to numerical integration methods on Euclidean spaces. The MKW Hopf algebra can be obtained by extending the Guin--Oudom construction to the free post-Lie algebra. 

Let $\mathcal{C}$ be a set, the free post-Lie algebra generated by $\mathcal{C}$ can be described by planar rooted trees whose vertices are decorated by elements from the set $\mathcal{C}$. Recall that a planar rooted tree is a rooted tree $\tau$ with vertex set $N(\tau)$ and an embedding in the plane. For example, the two following trees are isomorphic as decorated graphs but different as planar trees:
\begin{align*}
\Forest{[a[b][c[d]]]} \quad \text{and} \quad \Forest{[a[c[d]][b]]}.
\end{align*}
Let $\mathcal{PT}^{\mathcal{C}}$ denote the vector space spanned by planar trees decorated by $\mathcal{C}$. Define the left grafting product by letting $\tau_1 \graft \tau_2$ denote the sum over all ways to graft the root of the planar rooted tree $\tau_1$ onto a vertex of the planar rooted tree $\tau_2$ such that $\tau_1$ is leftmost in the planar embedding. For example:
\begin{align*}
	\Forest{[a[b]]}\graft \Forest{[c[d][e[f]]]}
	=\Forest{[c[a[b]][d][e[f]]]}
		+\Forest{[c[d[a[b]]][e[f]]]}
		+\Forest{[c[d][e[a[b]][f]]]}
		+\Forest{[c[d][e[f[a[b]]]]]}.
\end{align*}
Letting $Lie(\mathcal{PT}^{\mathcal{C}})$ denote the free Lie algebra generated by $\mathcal{PT}^{\mathcal{C}}$ and extending the left grafting operator by identities \eqref{postLie1} and \eqref{postLie2} 
\begin{align*}
	[\tau_1,\tau_2]\graft \tau_3
	&= \tau_1 \graft (\tau_2 \graft \tau_3) - (\tau_1 \graft \tau_2)\graft \tau_3 
				- \tau_2 \graft (\tau_1 \graft \tau_3) + (\tau_2 \graft \tau_1)\graft \tau_3,\\ 
	\tau_1 \graft [\tau_2,\tau_3]
	&= [\tau_1 \graft \tau_2,\tau_3] + [\tau_2,\tau_1 \graft \tau_3], 
\end{align*}
for all $\tau_1,\tau_2,\tau_3 \in \mathcal{PT}^{\mathcal{C}}$, gives the free post-Lie algebra. 

We now apply the Guin--Oudom construction at the level of post-Lie algebra \cite{EbrahimiFardLundervoldMuntheKaas2015}. Note that the universal enveloping algebra of $Lie(\mathcal{PT}^{\mathcal{C}})$ can be described by \textit{ordered forests} $\mathcal{OF}^{\mathcal{C}}$, i.e., ordered sequences of planar rooted trees. The extension of the left grafting product to ordered forests is given by letting $\omega_1 \graft \omega_2$ be the sum over all ways to leftmost graft all roots of $\omega_1$ onto vertices in $\omega_2$ such that if two roots in $\omega_1$ are grafted onto the same vertex, then their planar order is preserved. For example:
\begin{align*}
	\Forest{[a[b]]}\Forest{[c]} \graft \Forest{[d[e]]}\Forest{[f]}
	=\Forest{[d[a[b]][c][e]]}\Forest{[f]}+\Forest{[d[a[b]][e[c]]]}\Forest{[f]}
		+\Forest{[d[a[b]][e]]}\Forest{[f[c]]}+\Forest{[d[c][e[a[b]]]]}\Forest{[f]}
		+\Forest{[d[e[a[b]][c]]]}\Forest{[f]}+\Forest{[d[e[a[b]]]]}\Forest{[f[c]]}
		+\Forest{[d[c][e]]}\Forest{[f[a[b]]]}+\Forest{[d[e[c]]]}\Forest{[f[a[b]]]}
		+\Forest{[d[e]]}\Forest{[f[a[b]][c]]}.
\end{align*}

The structure $(\mathcal{OF}^{\mathcal{C}},\graft,\cdot)$ is that of the free D-algebra \cite{Munthe-KaasWright2008}. The associative product defined on $\mathcal{OF}^{\mathcal{C}}$ by
\begin{align}
\label{GLprod}
	A \ast B = A_{(1)} \cdot (A_{(2)}\graft B)
\end{align}
is called the planar Grossman--Larson product, and the Hopf algebra $\mathcal{H}_{\scriptscriptstyle{\text{GL}}}=(\mathcal{OF}^{\mathcal{C}},\ast,\Delta_{\shuffle})$ is called the planar Grossman--Larson Hopf algebra. For example
\begin{align*}
	\tau_1 \ast \tau_2 
	&= \tau_1 \tau_2 + \tau_1 \graft \tau_2\\  
	\tau_1 \ast \tau_2  \ast \tau_3 
	&= (\tau_1 \ast \tau_2) \tau_3 + \tau_1 (\tau_2 \graft \tau_3) + \tau_2 (\tau_1 \graft \tau_3) + (\tau_1 \ast \tau_2) \graft \tau_3\\
	&= \tau_1 \tau_2 \tau_3 + (\tau_1 \graft \tau_2) \tau_3
		+ \tau_1 (\tau_2 \graft \tau_3) + \tau_2 (\tau_1 \graft \tau_3) 
		+ (\tau_1\tau_2) \graft \tau_3 + (\tau_1 \graft \tau_2) \graft \tau_3\\
	&= \tau_1 \tau_2 \tau_3 + (\tau_1 \graft \tau_2) \tau_3
		+ \tau_1 (\tau_2 \graft \tau_3) + \tau_2 (\tau_1 \graft \tau_3) 
		+ \tau_1 \graft (\tau_2 \graft \tau_3).	
\end{align*}
Observe that the antipode of $\mathcal{H}_{\scriptscriptstyle{\text{GL}}}$, which we denote $S_*$, is different from the original antipode $S$. The Grossman--Larson product \eqref{GLprod} can be inverted
\begin{align}
\label{GLprodinv1}
	A \cdot B = A_{(1)} \ast (S_*(A_{(2)})\graft B),
\end{align}
from which we deduce the recursion
\begin{align}
%	A \cdot \tau_{n} 
%	&= A * \tau_{n} 
%		+ A^{(1)} \ast \big(S_*(A^{(2)}) \graft \tau_{n}\big), \label{GLprodinv2}\\
	\tau_1 \cdot A 
	&=  \tau_{1} * A  - \tau_1 \graft A. \label{GLprodinv2}
\end{align}
Using $A \graft 1 = \epsilon(A)1$ and $1 \graft A = A$, for any $A \in \mathcal{H}_{\scriptscriptstyle{\text{GL}}}$, one can show the formula  \cite{Li2022postHopf}
\begin{align}
	S_*(A)	
	&=S_*(A_{(1)}) \graft S(A_{(2)}), 		\nonumber
\end{align}
which implies the recursion 
\begin{align}
	S_*(A)
	&=S(A) + S_*(A^{(1)}) \graft S(A^{(2)}). 	\label{GLantipode}
\end{align}
For example 
\begin{align}
	S_*(\tau) 
	&= S(\tau) = -\tau \nonumber \\
	S_*(\tau_1\tau_2) 
	&= S(\tau_1\tau_2) + \tau_1 \graft \tau_2 + \tau_2 \graft \tau_1 = \tau_2 \ast\tau_1 + \tau_1 \graft \tau_2 .\nonumber
\end{align}
The MKW Hopf algebra $\mathcal{H}_{\scriptscriptstyle{\text{MKW}}}=(\mathcal{OF}^{\mathcal{C}},\shuffle,\Delta_{\scriptscriptstyle{\text{MKW}}})$ is the graded dual of $\mathcal{H}_{\scriptscriptstyle{\text{GL}}}$. The Munthe-Kaas--Wright coproduct, $\Delta_{\scriptscriptstyle{\text{MKW}}}$, can be described by left-admissible cuts \cite{Munthe-KaasWright2008}. Let $\tau$ be a planar rooted tree, and $c$ be a subset of edges in $\tau$. We say that $c$ is left-admissible if:
\begin{enumerate}
\item For each edge $e \in c$, every other edge outgoing from the same vertex as $e$ and to the left of $e$ in the planar embedding, is also in $c$.
\item Every path from the root of $\tau$ to any of its leaves, contains at most one edge in $c$.
\end{enumerate}
For $c$ a left-admissible cut, we denote the trunk $T^c(\tau)$, which is the connected component of $\tau$ that contains the root after the edges of $c$ have been removed. The pruned part, denote by $P^c(\tau)$, is made of the subtrees that result from deleting the edges in $c$. Note that among those subtrees that are cut off from the same vertex we keep the planar order. As a last step, the resulting forests that are cut off from different vertices are shuffled together to form $P^c(\tau)$. We can then describe $\Delta_{\scriptscriptstyle{\text{MKW}}}$ on planar rooted trees by:
\begin{align*}
	\Delta_{\scriptscriptstyle{\text{MKW}}}(\tau)
	=\sum_{c \text{ left-adm.~cut}} P^{c}(\tau) \otimes T^c(\tau).
\end{align*}
For example:
\begin{align*}
	\Delta_{\scriptscriptstyle{\text{MKW}}}(\Forest{[a[b][c[d]]]})
	=1 \otimes \Forest{[a[b][c[d]]]} 
		+ \Forest{[b]} \otimes \Forest{[a[c[d]]]} 
		+ \Forest{[d]} \otimes \Forest{[a[b][c]]} 
		+ \Forest{[b]}\shuffle \Forest{[d]} \otimes \Forest{[a[c]]} 
		+ \Forest{[b]}\Forest{[c[d]]} \otimes \Forest{[a]} 
		+ \Forest{[a[b][c[d]]]} \otimes 1.
\end{align*}
Now let $B_+ : \mathcal{OF}^{\mathcal{C}} \to \mathcal{PT}^{\mathcal{C}}$ denote the map given by grafting all trees in the input forest onto a single root, for example:
\begin{align*}
	B_+(\Forest{[a[b]]}\Forest{[c]})=\Forest{[[a[b]][c]]},
\end{align*}
and let $B_- : \mathcal{PT}^{\mathcal{C}} \to \mathcal{OF}^{\mathcal{C}}$ denote the inverse map. We may add a decoration, $B_+^a$, $a \in {\mathcal{C}}$, to indicate the decoration of the new root. Then $\Delta_{\scriptscriptstyle{\text{MKW}}}$ applied to a forest $w \in \mathcal{OF}^{\mathcal{C}}$ can be described by:
\begin{align*}
	\Delta_{\scriptscriptstyle{\text{MKW}}}(\omega)
	=( \mathrm{id} \otimes B_-)(\Delta_{\scriptscriptstyle{\text{MKW}}}(B_+(\omega)) - B_+(\omega)\otimes 1 ).
\end{align*}
Note that the decoration of the added root, $B_+(\omega)$, plays no role in this identity as the root gets removed by $B_-$. This identity is saying that $\Delta_{\scriptscriptstyle{\text{MKW}}}$ applied to a forest will perform left-admissible cuts on all trees in the forest, where a left-admissible cut can also cut between the different trees in the forest. For example:
\begin{align*}
	\Delta_{\scriptscriptstyle{\text{MKW}}}(\Forest{[a[b]]}\Forest{[c[d]]}\Forest{[f]})
	&=1 \otimes \Forest{[a[b]]}\Forest{[c[d]]}\Forest{[f]} 
		+ \Forest{[b]} \otimes \Forest{[a]}\Forest{[c[d]]}\Forest{[f]} 
		+ \Forest{[d]} \otimes \Forest{[a[b]]}\Forest{[c]}\Forest{[f]} 
		+ \Forest{[a[b]]} \otimes \Forest{[c[d]]}\Forest{[f]}\\
	&\quad + \Forest{[b]}\shuffle \Forest{[d]} \otimes \Forest{[a]}\Forest{[c]}\Forest{[f]}
		+\Forest{[a[b]]}\shuffle \Forest{[d]} \otimes \Forest{[c]}\Forest{[f]}
		+\Forest{[a[b]]}\Forest{[c[d]]} \otimes \Forest{[f]} 
		+\Forest{[a[b]]}\Forest{[c[d]]}\Forest{[f]} \otimes 1.
\end{align*}

\section{Hopf algebras with a natural growth operator} 
\label{sect:Foissy}

Following Foissy \cite{Foissy2002,FoissyFrench}, we introduce the notion of natural growth operator on a connected graded Hopf algebra $\mathcal{H}$. We will then show that the MKW Hopf algebra has a natural growth operator, which allows us to apply Foissy's results. Note that we employ Sweedler's notation for both the coproduct, $\Delta(x) = x_{(1)} \otimes x_{(2)}$, as well as the reduced coproduct on $\mathcal{H}$
$$
	\hat{\Delta}(x) = x^{(1)} \otimes x^{(2)} := \Delta(x) - x \otimes 1 - 1 \otimes x.
$$
For more details on combinatorial bialgebras, including examples and applications, we refer the reader to the recent textbook \cite{CP2021} and the survey article \cite{Manchon2008}.

\begin{definition}
\label{def:naturalgrowth}
Let $\mathcal{H}$ be a connected graded Hopf algebra with coproduct $\Delta: \mathcal{H} \to \mathcal{H} \otimes \mathcal{H}$. If the bilinear operation $\top: \mathcal{H} \otimes \mathcal{H} \to \mathcal{H}$ satisfies the equation 
\begin{equation}
\label{eq::NaturalGrowth}
	\hat{\Delta}(x \top y) = x \otimes y + x^{(1)} \otimes (x^{(2)} \top y ),
\end{equation}
for $x \in \mathcal{H}$ and every primitive element $y \in P(\mathcal{H}):=\{x \in \mathcal{H}\ |\ \hat{\Delta}(x)=0 \} \subset  \mathcal{H}$, then we say that $\top$ is a natural growth operator for the Hopf algebra $\mathcal{H}$.
\end{definition}

Let $p_1,\ldots,p_i$ be primitive elements in the Hopf algebra $(\mathcal{H},\top)$ with a natural growth operator. As the latter is non-associative, we shall use the notation:
\begin{align*}
	p_i \top \cdots \top p_1 := (p_i \top \cdots \top p_2) \top p_1.
\end{align*}

\begin{lemma} [\cite{Foissy2002}]  
\label{lemma::Deconcatenation}
Let $p_1,\ldots,p_i$ be primitive elements in the Hopf algebra $(\mathcal{H},\top)$ with natural growth operator, then:
\begin{align*}
	\hat{\Delta}(p_i \top \cdots \top p_1)
	=\sum_{j=1}^{i-1}(p_i \top \cdots \top p_{j+1}) \otimes (p_j \top \cdots \top p_1).
\end{align*}
\end{lemma}

\begin{proof}
It was observed in \cite{Foissy2002} that this intriguing property follows directly from identity \eqref{eq::NaturalGrowth}.
\end{proof}

\begin{proposition}  [\cite{Foissy2002}]  \label{prop::Projection}
\label{Prop::Projection}
A projection  $\pi: \mathcal{H} \to P(\mathcal{H})$ onto the set of primitive elements $P(\mathcal{H})$ can be computed recursively:
\begin{align*}
	\pi(x)=x - x^{(1)} \top \pi(x^{(2)}).
\end{align*}
\end{proposition}

\begin{proof}
The proof given in \cite{Foissy2002} in the specific context of the BCK Hopf algebra $\mathcal{H}_{\scriptscriptstyle{\text{BCK}}}$ of non-planar rooted trees applies here. The equality $\hat{\Delta}(\pi(x))=0$ is clear if $x$ has degree one. Suppose that the equality holds for all elements up to degree $k$ and let $y$ be an element of degree $k+1$. Then:
\begin{align*}
	\hat{\Delta}(y)
	&= y^{(1)} \otimes y^{(2)}\\
	&= y^{(1)} \otimes \big(\pi(y^{(2)}) + (y^{(2)})^{(1)} \top \pi( (y^{(2)})^{(2)}) \big).
\end{align*}
Here the induction hypothesis applies because $y^{(2)}$ has degree at most $k$. Using coassociativity, we obtain:
\begin{align*}
	\hat{\Delta}(y)
	=y^{(1)}\otimes \pi(y^{(2)})+(y^{(1)})^{(1)}\otimes \big((y^{(1)})^{(2)}\top \pi(y^{(2)}) \big).
\end{align*}
The natural growth property \eqref{eq::NaturalGrowth} implies then that $\hat{\Delta}(y)=\hat{\Delta}(y^{(1)}\top \pi(y^{(2)}))$. Hence, the difference, $\pi(y)=y-y^{(1)}\top \pi(y^{(2)})$, is primitive. The statement follows by mathematical induction.
\end{proof}

\begin{example}
\label{ex:piBCK}
As mentioned before, in \cite{Foissy2002} the primitive elements in the BCK Hopf algebra $\mathcal{H}_{\scriptscriptstyle{\rm{BCK}}}$ are given in terms of the projection $\pi$. A few examples are in order.
$$
\pi(\Forest{[]}) =\Forest{[]}
\qquad	
\pi(\Forest{[[]]}) =0
\qquad
\pi(\Forest{[]}\Forest{[]}) = \Forest{[]}\Forest{[]} - 2\Forest{[[]]}
\qquad
\pi(\Forest{[[[]]]}) = \pi(\Forest{[[][]]}) = \pi(\Forest{[]}\Forest{[[]]}) = 0
\qquad
\pi(\Forest{[]}\Forest{[]}\Forest{[]})=\Forest{[]}\Forest{[]}\Forest{[]}-3\Forest{[]}\Forest{[[]]}+3\Forest{[[[]]]}.
$$
See Remark \ref{rmk:BCK} below, where we will further elaborate on the precise definition of $\top$, in the context of the MKW Hopf algebra -- which also clarifies its definition in the BCK case. 
\end{example}

\begin{example}
\label{ex:shuffle}
Consider the alphabet $A$ and the shuffle Hopf algebra $\mathcal{H}_{\shuffle}=(T(A),\shuffle,\Delta)$, with deconcatenation as coproduct \cite{Reutenauer1993}. Defining $\top$ to be the concatenation product, gives a natural growth operator:
\begin{align*}
	\hat{\Delta}(w \top x)=\hat{\Delta}(wx) = w \otimes x + w^{(1)} \otimes w^{(2)} \top x,
\end{align*}
whenever $x$ is a primitive element, i.e., a letter in $A$. Indeed, for $a \in A$ we compute 
\allowdisplaybreaks
\begin{align*}
	\hat{\Delta}(w \top a)
	&= \Delta(wa) - \mathbf{1} \otimes wa - wa \otimes \mathbf{1} \\
	&=(w \otimes \mathrm{id}) \Delta(a) + \Delta(w)( \mathrm{id} \otimes a) 
		- w \otimes a - \mathbf{1} \otimes wa - wa \otimes \mathbf{1}\\ 
	&= w \otimes a + wa \otimes \mathbf{1} + \Delta(w) (\mathrm{id} \otimes a)
		- w \otimes a - \mathbf{1} \otimes wa - wa \otimes \mathbf{1}\\
	&=\Delta(w)(\mathrm{id} \otimes a) - \mathbf{1} \otimes wa\\
	&=\big(\sum_{i=0}^n w_{i} \otimes w_{n-i}\big)(\mathrm{id}\otimes a)  - \mathbf{1} \otimes wa\\
	&=\big(\sum_{i=1}^n w_{i} \otimes w_{n-i}\big)(\mathrm{id} \otimes a) \\
	&= w \otimes a + \big(\sum_{i=1}^{n-1} w_{i} \otimes w_{n-i}\big)(\mathrm{id} \otimes a) \\
	&= w \otimes a + w^{(1)} \otimes w^{(2)}a \\
	&= w \otimes a + w^{(1)} \otimes (w^{(2)} \top a).
\end{align*}

Lemma \ref{lemma::Deconcatenation} just describes the usual definition of deconcatenation on the tensor algebra $T(A)$. The projection $\pi$ defined in Proposition \ref{Prop::Projection} is the identity map on letters and sends words of length greater than one to zero. 
\end{example}

We remark that for words $w,v \in T(A)$, deconcatenation and concatenation satisfy 
\begin{align}
	\Delta(wv)=(w \otimes \mathrm{id}) \Delta(v) 
				+ \Delta(w)( \mathrm{id} \otimes v) 
				- w \otimes v. \label{eq::InfinitesimalBialgebra}
\end{align}
This identity is known to define a unital infinitesimal bialgebra \cite{Aguiar2000b,Aguiar2000a}. In this respect we remark that a quadruple $(\mathcal{A},\ast,\odot,\Delta)$ is called a 2-associative bialgebra if $(\mathcal{A},\ast,\Delta)$ is a bialgebra and $(\mathcal{A},\odot,\Delta)$ is a unital infinitesimal bialgebra. By letting the second argument in \eqref{eq::InfinitesimalBialgebra} be primitive, we see that any 2-associative Hopf algebra is a Hopf algebra with a natural growth operator. It was shown in \cite{LodayRonco2004} that 2-associative Hopf algebras are exactly the cofree Hopf algebras, which are in bijection with $B_{\infty}$-algebras. See also \cite{FoissyMalvenutoPatras2014}. Note that being 2-associative is a stricter requirement than having a natural growth operator, since a natural growth operator does not need to be associative. However a natural growth operator can be used to define a 2-associative Hopf algebra.

\begin{lemma} \label{lemma::FiDirect}
For the maps $F_i: P(\mathcal{H})^{\otimes i} \to \mathcal{H}$ defined by
$$
	F_i(p_i \otimes \cdots \otimes p_1) := p_i \top \cdots \top p_1
$$	
the following holds
\begin{align*}
	\hat{\Delta}^{i-1}\circ F_i 
	&= \mathrm{id}_{P(\mathcal{H})^{\otimes i}},\\
	\hat{\Delta}^{i+k} \circ F_i
	&= 0, \qquad\ k \geq 0. 
\end{align*}
Furthermore, each $F_i$ is injective and the sum
\begin{align*}
\mathbb{K}1 + \sum_{i=1}^{\infty}\text{Im}(F_i)
\end{align*}
is direct.
\end{lemma}

\begin{proof}
This is shown by Foissy in \cite{Foissy2002} to follow as a consequence of Lemma \ref{lemma::Deconcatenation}.
\end{proof}

%%%%%%%%%%%%%%%%%%%%%%%%%%%%%%%%%%%

\subsection{Characterisation of finite-dimensional comodules}
\label{ssec:comod}

Let $\mathcal{H}=(V,\cdot,\Delta)$ be a Hopf algebra with natural growth operator $\top$. Its graded dual Hopf algebra is denoted by $\mathcal{H}^{\ast}=(V^{\ast},\star, \Delta_{\cdot})$. We first define a family of comodules over $\mathcal{H}$, and show that every finite-dimensional comodule belongs to this family. To define this family, we need the following notation. Let $(i,j) \in \mathbb{N}^2$, with $i \leq j$. Denote $I_{i,j} := \{i,i+1,\ldots, j\}$. A decomposition of $I_{i,j}$ is a partition into connected parts, denoted $I_{i_1,j_1}, \ldots , I_{i_k,j_k}$ with $i=i_1 \leq j_1 \leq i_2 \leq j_2 \leq \cdots \leq i_k \leq j_k =j $. The set of decompositions of $I_{i,j}$ is denoted~$\mathcal{D}_{i,j}$.

\begin{proposition}
Let $n \geq 1$, and $(p_{i,j})_{1 \leq i \leq j \leq n}$ be any family of $\frac{n^2+n}{2}$ primitive elements of $\mathcal{H}$. Let $\mathcal{C}$ be a vector space of dimension $n+1$, with basis $(e_0,\ldots,e_n)$. Define:
\begin{align*}
	\Delta_C(e_0)
	&= 1 \otimes e_0,\\
	\Delta_C(e_i)
	&= \sum_{j=0}^{i-1}\Big[\sum_{I_{i_1,j_1},\ldots,I_{i_k,j_k}\in \mathcal{D}_{j+1,i}} 
	p_{i_k.j_k}\top \cdots \top p_{i_1,j_1}\otimes e_j  \Big]+1\otimes e_i.
\end{align*}
Then $(\mathcal{C},\Delta_C)$ is an $\mathcal{H}$-comodule, which we denote by $\mathcal{C}_{(p_{i,j})}$.
\end{proposition}

\begin{proof}
Foissy's proof in the context of the Butcher--Connes--Kreimer Hopf algebra \cite{Foissy2002} relies on Lemma \ref{lemma::Deconcatenation}, and hence applies here.
\end{proof}

\begin{theorem} 
\label{thm::EveryComoduleIsPij}
Let $(C,\Delta_C)$ be a finite-dimensional comodule. If the dimension of $C$ is one, then $C$ is trivial, that is to say $\Delta_C(x)=1 \otimes x$ for all $x \in C$. If the dimension of $C$ is $n \geq 2$, then there exists a family $(p_{i,j})_{1 \leq i \leq j \leq n}$ of primitive elements such that $C$ is isomorphic to $C_{(p_{i,j})}$
\end{theorem}

\begin{proof}
The proof follows by construction. See \cite{Foissy2002} for details.
\end{proof}

The family $(p_{i,j})_{1 \leq i \leq j \leq n}$ can be written as a strictly upper-triangular matrix $\mathcal{P}$:
\begin{align*}
\mathcal{P}=\begin{bmatrix}
0 	& p_{1,1} 	& \cdots & p_{1,n}\\
\vdots 	& \ddots 	& \ddots & \vdots \\
0 	& \cdots 	& \ddots & p_{n,n} \\
0 	& \cdots 	& \cdots & 0
\end{bmatrix}.
\end{align*}
We say that $(p_{i,j})$ is reduced of type $(c_0,\ldots,c_k)$ if the matrix $\mathcal{P}$ can be written in block form:
\begin{align*}
\mathcal{P}=\left[ \begin{array}{c|c|c|c}
0  	& \mathcal{P}_{1,1} 	& \cdots 	& \mathcal{P}_{1,n}\\ \hline
\vdots 	& \ddots  			& \ddots  	& \vdots \\ \hline
0 	& \cdots 			& \ddots	& \mathcal{P}_{n,n} \\\hline
0 	& \cdots 			& \cdots	& 0
\end{array} \right],
\end{align*}
where the diagonal zero blocks are of size $c_0 \times c_0, \ldots,c_k \times c_k$.

\begin{example}[\cite{Foissy2002}]
The matrix
\begin{align*}
\mathcal{P} = \left[ \begin{array}{c|cc|cc}
0 & a & b & x & y \\ \hline
0 & 0 & 0 & c & e \\
0 & 0 & 0 & d & f \\ \hline
0 & 0 & 0 & 0 & 0 \\
0 & 0 & 0 & 0 & 0 
\end{array}  \right]
\end{align*}
is reduced of type $(1,2,2)$.
\end{example}

For $\mathcal{C}$ an $\mathcal{H}$-comodule, define $\mathcal{C}_0=\{x \in \mathcal{C}\ |\ \Delta_C(x)=1 \otimes x  \}$, then inductively define $\mathcal{C}_{i+1} \subset \mathcal{C}$ by:
\begin{align*}
\mathcal{C}_i 
&\subset \mathcal{C}_{i+1},\\
\frac{\mathcal{C}_{i+1}}{\mathcal{C}_i}
&=\Bigl( \frac{\mathcal{C}}{\mathcal{C}_i} \Bigr)_0.
\end{align*}
If $\mathcal{C}$ is finite-dimensional, then Theorem \ref{thm::EveryComoduleIsPij} implies that $\mathcal{C}_0$ is non-zero. Hence this construction gives a flag of comodules: $\mathcal{C}_0 \subsetneq \cdots \subsetneq \mathcal{C}_k = \mathcal{C}$.

\begin{proposition} [\cite{Foissy2002}]
Let $(p_{i,j})$ be reduced of type $(c_0,\ldots,c_k)$ and let $(e_0,\ldots,e_n)$ be the basis of $\mathcal{C}_{(p_{i,j})}$. Then for all $\ell \in \{1,\ldots, k\}$, $(e_0,\ldots,e_{c_0+\cdots+c_{\ell-1}})$ is a basis of $(\mathcal{C}_{(p_{i,j})})_{\ell}$.
\end{proposition}

\begin{proposition} [\cite{Foissy2002}]
Let $\mathcal{C}$ be a finite-dimensional comodule with basis $(e_0,\ldots,e_n)$ such that $(e_0,\ldots,e_{\dim(\mathcal{C}_i)-1})$ is a basis for $\mathcal{C}_i$, for $0 \leq i \leq k$. Let $(p_{i,j})$ be the family of primitive elements constructed in the proof of Theorem \ref{thm::EveryComoduleIsPij}. Then $(p_{i,j})$ is reduced of type $(c_0,\ldots,c_k)$, where $c_0=\dim(\mathcal{C}_0), c_i=\dim(\mathcal{C}_i)-\dim(\mathcal{C}_{i-1})$.
\end{proposition}

\begin{corollary} [\cite{Foissy2002}]
For any finite-dimensional comodule $\mathcal{C}$, there exists a reduced family $(p_{i,j})$ such that $\mathcal{C}$ is isomorphic to $\mathcal{C}_{(p_{i,j})}$. \\
If $(p_{i,j}),(p'_{i,j})$ are reduced families such that $\mathcal{C}_{(p_{i,j})}$ and $\mathcal{C}_{(p'_{i,j})}$ are isomorphic, then $(p_{i,j})$ and $(p'_{i,j})$ have the same type.
\end{corollary}

For a comodule $\mathcal{C}$ we will call the type of any family $(p_{i,j})$, such that $\mathcal{C}$ is isomorphic to $\mathcal{C}_{(p_{i,j})}$ the type of $\mathcal{C}$. Let 
\begin{align*}
G_{(c_0,\ldots,c_k)}
&= \biggl\{ \left[ \begin{array}{c|c|c|c}
g_{0,0} 	& g_{0,1} 	& \cdots 	& g_{0,k} \\ \hline
\vdots 	& \ddots 	& \ddots 	& \vdots \\ \hline
0 		& \cdots 	& \ddots 	& g_{k-1,k} \\ \hline
0 		& \cdots 	& \cdots 	& g_{k,k}
\end{array} \right] \ |\ g_{i,i} \in GL_{c_i}(\mathbb{Q}) \biggr\}
\end{align*}
be a subgroup of $GL_{c_0+\cdots+c_k}(\mathbb{Q})$. It acts on reduced families of type $(c_0,\ldots,c_k)$ by $g.\mathcal{P}:=g\mathcal{P}g^{-1}$, where $\mathcal{P}$ is the matrix of the family.

\begin{proposition} [\cite{Foissy2002}]
Let $(p_{i,j})$ and $(p'_{i,j})$ be two reduced families of primitive elements, and $(c_0,\ldots,c_k)$ the type of $(p_{i,j})$. Then $C_{(p_{i,j})}$ is isomorphic to $C_{p'_{i,j}}$ if and only if $(p'_{i,j})$ is of type $(c_0,\ldots,c_k)$ and there exists a $g \in G_{(c_0,\ldots,c_k)}$ such that $g.\mathcal{P}=\mathcal{P}'$.
\end{proposition}

\begin{theorem} [\cite{Foissy2002}]
Let $\mathcal{P}_{(c_0,\ldots,c_k)}$ be the set of reduced families of type $(c_0,\ldots,c_k)$ and $\mathcal{O}_{(c_0,\ldots,c_k)}$ the orbit space under the action of $G_{(c_0,\ldots,c_k)}$. Then there is a bijection from $\mathcal{O}_{(c_0,\ldots,c_k)}$ to the set of isomorphism classes of $\mathcal{H}$-comodules of type $(c_0,\ldots,c_k)$. There is a bijection from the disjoint union of all $\mathcal{O}_{(c_0,\ldots,c_k)}$ into the set of finite-dimensional comodules.
\end{theorem}

%%%%%%%%%%%%%%%%%%%%%%%%%%%%%%%%%%%

\subsection{A second grading on $\mathcal{H}$}
\label{ssec:2ndgrading}

The Hopf algebra $\mathcal{H}=(V,\cdot,\Delta)$ with natural growth operator $\top$ has the infinite-dimensional comodule $\mathcal{C}=(V,\Delta)$ obtained from the Hopf algebra structure. Although infinite-dimensional, we can write it as a union of finite-dimensional comodules $\mathcal{C}_i$.

\begin{proposition}
$\mathcal{C}_0 = \mathbb{K}1$, $\mathcal{C}_i = \mathbb{K}1 \oplus \bigoplus_{j=1}^i \text{Im}(F_j)$, and $\mathcal{C}= \mathbb{K}1 \oplus \bigoplus_{j=1}^{\infty} \text{Im}(F_j)$.
\end{proposition}

The last proposition tells us that any element $x \in \mathcal{H}$ can be written as:
\begin{align*}
	x=\sum_{i=1}^n p^i_{k_i} \top \cdots \top p^i_1,
\end{align*}
for some primitive elements $p^i_1,\ldots,p^i_{k_i}$ and some $n$.Together with Lemma \ref{lemma::FiDirect}, this gives us that every Hopf algebra with natural growth is cofree.

\begin{proposition} [\cite{Foissy2002}]
\label{prop::FreeAssociative}
The graded dual $\mathcal{H}^{\ast}=(V^{\ast},\star,\Delta_{\star})$ is a free associative algebra. Its primitive elements form a free Lie algebra.
\end{proposition}

Let the primitive degree $\deg: \mathcal{H} \to \mathbb{N}$ be defined by
\begin{align}
	\deg(x)
	&:= \inf \{k\ |\ \exists p_1,\ldots,p_k \in P(\mathcal{H})(x= p_k \top \cdots \top p_1 )  \} \nonumber \\
	&= \inf\{k\ |\ \hat{\Delta}^k(x)=0\}. \label{pdegree}
\end{align}
Note that this gives a filtration for the Hopf algebra, which we denote $|x|$, and that is different from the original grading. We now define a second Hopf algebra, denoted $gr(\mathcal{H})=(V,\shuffle,\Delta)$, over the same vector space. As an algebra, $gr(\mathcal{H})$ is the graded algebra associated to the filtration $\deg$. Recall that $V=\mathbb{K}1 \oplus \bigoplus_{j=1}^{\infty}Im(F_j)$ and note that $Im(F_j)$ are the homogeneous components with respect to the grading. The product and coproduct are given by:
$$
	(p_{j + \ell} \top \cdots \top p_{j+1}) \shuffle (p_j \top \cdots \top p_1 )
	:= \sum_{\sigma(j,\ell)-\text{shuffle}} p_{\sigma(j+\ell)} \top \cdots \top p_{\sigma(1)},
$$
respectively
$$
	\Delta(p_j \top \cdots \top p_1)
	:= (p_j \top \cdots \top p_1 \otimes 1) + (1 \otimes p_j \top \cdots \top p_1) 
	+ \sum_{k=2}^{j}(p_j \top \cdots \top p_k) \otimes (p_{k-1}\top \cdots \top p_1),
$$
where a $\sigma(j,\ell)-$shuffle is a permutation of $\{1,\ldots,j+\ell\}$ such that $\sigma(1) < \cdots < \sigma(j)$ and $\sigma(j+1)< \cdots < \sigma(j+\ell)$.

\begin{remark}
Although $\top$ is in general not an associative product, the Hopf algebra $gr(\mathcal{H})$ looks rather similar to the tensor Hopf algebra over the primitive elements.
\end{remark}

\begin{remark}
The identity map $\mathrm{id}: V \to V$ is a coalgebra morphism $\mathcal{H} \to gr(\mathcal{H})$, but not an algebra morphism. We will see later that the two Hopf algebras are isomorphic whenever $\mathcal{H}$ is commutative, but via a different map.
\end{remark}

Following \cite{Foissy2002}, we now classify all coalgebra endomorphisms of $\mathcal{H}$.

\begin{proposition}[\cite{Foissy2002}] \label{prop::CoalgebraMorphism}
Let $(u_i)_{i \in \mathbb{N}}$ be a family of maps $P(\mathcal{H})^{\otimes i} \to P(\mathcal{H})$, and define the linear map $\Phi_{(u_i)}$ by:
\begin{align*}
\Phi_{(u_i)}(1)
	&= 1,\\
	\Phi_{(u_i)}(p_n \top \cdots \top p_1) 
	&= \sum_{k=1}^n \sum_{a_1+\cdots +a_k=n} 
	F_k\big((u_{a_1}\otimes \cdots \otimes u_{a_k}) (F_n^{-1}(p_n \top \cdots \top p_1))\big).
\end{align*}
Then $\Phi_{(u_i)}$ is a coalgebra morphism of $\mathcal{H}$ (and of $gr(\mathcal{H})$). Furthermore, if $\Phi$ is any coalgebra morphism of $\mathcal{H}$, then there exists a family $(u_i)_{i \in \mathbb{N}}$ such that $\Phi=\Phi_{(u_i)}$.
\end{proposition}

\begin{proposition} [\cite{Foissy2002}]
\label{prop::BijectiveCoalgMorphism}
The coalgebra morphism $\Phi_{(u_i)}$ is bijective if and only if $u_1: P(\mathcal{H}) \to P(\mathcal{H})$ is bijective.
\end{proposition}

\begin{proposition} [\cite{Foissy2002}]
\label{prop::BialgebraMorphism}
Let $\Phi=\Phi_{(u_i)}$ be a coalgebra endomorphism. Denote by $(u_i^n)_{i \in \mathbb{N}}$ the family defined by $u_i^n=u_i$ if $i \leq n$ and $u_i^n=0$ if $n<i$. Then:
\begin{enumerate}
\item $\Phi_{(u_i)}$ is a bialgebra endomorphism of $\mathcal{H}$ if and only if:
\begin{align*}
	u_{i+j}(F_{i+j}^{-1}(x_i \shuffle x_j)  )
	=-\Phi_{(u_i^{i+j-1})}(x_i \cdot x_j)+\Phi_{(u_i^{i+j-1})}(x_i)\cdot \Phi_{(u_i^{i+j-1})}(x_j),
\end{align*}
for all $x_i \in Im(F_i)$ and all $x_j \in Im(F_j)$.
\item $\Phi_{(u_i)}$ is a bialgebra endomorphism of $gr(\mathcal{H})$ if and only if:
\begin{align*}
	u_{i+j}(F_{i+j}^{-1}(x_i \shuffle x_j)  )
	=-\Phi_{(u_i^{i+j-1})}(x_i \shuffle x_j)+\Phi_{(u_i^{i+j-1})}(x_i)\shuffle \Phi_{(u_i^{i+j-1})}(x_j),
\end{align*}
for all $x_i \in Im(F_i)$ and all $x_j \in Im(F_j)$.
\end{enumerate}
\end{proposition}

\begin{remark}
Consider the shuffle algebra $\mathcal{H}_{\shuffle}$, then $P(\mathcal{H}_{\shuffle})^{\otimes i}$ can be identified with words of length $i$. We look at a coalgebra endomorphism $\Phi$ and examine the condition to be a bialgebra morphism. First look at a word $w=ab$ which is of length two, then
\begin{align*}	
	u_1(a)\shuffle u_1(b)
	&=\Phi(a)\shuffle \Phi(b)\\
	&=\Phi(a\shuffle b)\\
	&=u_2(a \shuffle b) + u_1(a)u_1(b)+u_1(b)u_1(a)\\
	&=u_2(a \shuffle b) + u_1(a)\shuffle u_1(b),
\end{align*}
tells us that $u_2$ vanishes on products. Then one can use an inductive argument to show that all $u_i$ vanishes on products, hence classifying all endomorphisms by maps from the dual primitives to the primitives.
\end{remark}

Let $\mathcal{M}$ denote the augmentation ideal, which is the same for both $\mathcal{H}$ and $gr(\mathcal{H})$. Both Hopf algebras are graded by $|\cdot |$, denote the decomposition with respect to this grading by:
\begin{align*}
	\mathcal{M}
	&= \bigoplus_i \mathcal{M}_i,\\
	\mathcal{M}^2
	&= \bigoplus_i \mathcal{M}_i^2.
\end{align*}

Note that:
\begin{align*}
\mathcal{M}_i=\mathcal{M}_i^2 + \sum_j \mathcal{H}_i\cap Im(F_j).
\end{align*}
Hence we can choose $K_{i,j}\subset \mathcal{H}_i \cap Im(F_j)$ such that:
\begin{align*}
\mathcal{M}_i = \mathcal{M}_i^2 \oplus \bigoplus_j K_{i,j}.
\end{align*}
Denote $K= \bigoplus_{i,j}K_{i,j}$.

\begin{lemma}[\cite{Foissy2002}]
$K$ generates the algebra $(V,\cdot )$.
\end{lemma}

\begin{lemma}[\cite{Foissy2002}]
$K$ generates the algebra $(V,\shuffle )$.
\end{lemma}

\begin{corollary}[\cite{Foissy2002}] \label{cor::ShuffleIso}
If $(V,\cdot)$ is commutative, then there exists an algebra morphism between $(V,\cdot )$ and $(V,\shuffle)$ that is the identity on $K$. This map is furthermore a Hopf algebra isomorphism between $\mathcal{H}$ and $gr(\mathcal{H})$.
\end{corollary}

\begin{remark}
Corollary \ref{cor::ShuffleIso} says that every commutative cofree Hopf algebra is isomorphism to a shuffle Hopf algebra.
\end{remark}

%%%%%%%%%%%%%%%%%%%%%%%%%%%%%%%%%%%

\subsection{Cocycles}
\label{ssec:cocycles}

Recall the definition of the $B_+$-operator in the context of the BCK Hopf algebra of non-planar rooted trees. It maps a forest of rooted trees, $\tau_1 \cdots \tau_n$, $\tau_i \in \mathcal{T}$, to a single tree by grafting the roots of each tree $\tau_i$ in the forest to a new root, that is, $B^{\scalebox{0.7}{\Forest{[]}}}_+(\tau_1 \cdots \tau_n)=\tau \in \mathcal{T}$. Note that $|B^{\scalebox{0.7}{\Forest{[]}}}_+(\tau_1 \cdots \tau_n)|=1 + \sum_{i=1}^n |\tau_i|$.
$$
B^{\scalebox{0.7}{\Forest{[]}}}_+(1)=\Forest{[]}
\qquad
B^{\scalebox{0.7}{\Forest{[]}}}_+(\Forest{[]})=\Forest{[[]]}
\qquad
B^{\scalebox{0.7}{\Forest{[]}}}_+(\Forest{[]}\Forest{[]})=\Forest{[[][]]}
\qquad
B^{\scalebox{0.7}{\Forest{[]}}}_+(\Forest{[]}\Forest{[[]]})=\Forest{[[][[]]]}.
$$ 
It satisfies the so-called Hochschild-1-cocycle identity
$$
	\Delta_{\scriptscriptstyle{\text{BCK}}} B^{\scalebox{0.7}{\Forest{[]}}}_+ 
	= B^{\scalebox{0.7}{\Forest{[]}}}_+ \otimes 1 
	+ (\mathrm{id} \otimes B^{\scalebox{0.7}{\Forest{[]}}}_+) \Delta_{\scriptscriptstyle{\text{BCK}}}, 
$$
which determines the coproduct uniquely as a morphism of unital algebras. Recall that the single vertex tree, $\Forest{[]}$, is primitive in the BCK Hopf algebra. For $\tau$ a forest, we note that $\tau \top \Forest{[]}:=B^{\scalebox{0.7}{\Forest{[]}}}_+(\tau)$ satisfies the natural growth identity \eqref{eq::NaturalGrowth}. Indeed, as the linear map $(\mathrm{id} - \eta \circ \epsilon)$ is the identity on non-empty forests and zero else, we check that
$$
	\hat{\Delta}_{\scriptscriptstyle{\text{BCK}}} (\tau \top \Forest{[]}) 
	= \big((\mathrm{id} - \eta \circ \epsilon) \otimes \cdot \top \Forest{[]} \big)(\tau \otimes 1 
	+ \tau^{(1)} \otimes \tau^{(2)})
	= \tau \otimes \Forest{[]} + \tau^{(1)} \otimes (\tau^{(2)}\top \Forest{[]}).
$$
Observe that identity \eqref{eq::NaturalGrowth} continues to hold if we replace \Forest{[]} in $\cdot \top \Forest{[]}$ by another primitive element in the BCK Hopf algebra. 

More generally, let $p$ be a primitive element in the Hopf algebra $(\mathcal{H},\top)$ with a natural growth operator, and define the operator $B_+^p: \mathcal{H} \to \mathcal{H}$ by:
\begin{align*}
	B_+^p(x)=x \top p.
\end{align*}
Regarding the primitive degree defined in \eqref{pdegree}, we see that 
$$
	\deg(B^p_+(x))= 1+ \deg(x).
$$

Note that the following proposition does not appear in \cite{Foissy2002}.

\begin{proposition}
The operator $B_+^p$ satisfies the cocycle property:
\begin{align} \label{eq::CoCycleIdentity}
	\Delta( B_+^p(x) )=B_+^p(x) \otimes 1 + (\mathrm{id} \otimes B_+^p)\Delta(x).
\end{align}
\end{proposition}

\begin{proof}
Since $p$ is primitive, we can use identity \eqref{eq::NaturalGrowth}:
\begin{align*}
	\Delta( B_+^p(x) )
	&= 1 \otimes B_+^p(x) + B_+^p(x) \otimes 1 + \hat{\Delta}(x\top p)\\
	&=1 \otimes B_+^p(x) + B_+^p(x) \otimes 1 + x \otimes p + x^{(1)} \otimes x^{(2)} \top p \\
	&= B_+^p(x) \otimes 1 + x_{(1)} \otimes x_{(2)} \top p \\
	&= B_+^p(x) \otimes 1 + (\mathrm{id}  \otimes B_+^p)\Delta(x).
\end{align*}
\end{proof}

\begin{remark}
\label{rmk:foissyfrench}
In \cite{FoissyFrench}, Foissy studied coalgebras endowed with maps $B_p^+$ satisfying equation \eqref{eq::CoCycleIdentity}. He showed that if a coalgebras has such a map for each primitive element, and furthermore satisfies some additional properties (that follow automatically if the coalgebra is furthermore a Hopf algebra), then the coalgebra is isomorphic to the tensor coalgebra.
\end{remark}

Recall that for a connected graded Hopf algebra the antipode can be computed recursively \cite{Manchon2008}
$$
	S(x)=-x - S(x^{(1)}) x^{(2)} = -x - x^{(1)} S(x^{(2)}). 
$$

\begin{corollary}[\cite{Foissy2002}]
Let $S$ be the antipode of $\mathcal{H}$, then:
\begin{align*}
	S(B_+^p(x))
	&=-B_+^p(x)-S(x)p-S(x^{(1)})B_+^p(x^{(2)})\\
	&=-B_+^p(x) + xp -x^{(1)}S(B_+^p(x^{(2)})).
\end{align*}
\end{corollary}

%%%%%%%%%%%%%%%%%%%%%%%%%%%%%%%%%%%
%%%%%%%%%%%%%%%%%%%%%%%%%%%%%%%%%%%

\section{Natural growth in the MKW Hopf algebra}

We are now ready to construct a natural growth operator on the Hopf algebra $\mathcal{H}_{\scriptscriptstyle{\text{MKW}}}$. Define the operator 
$$	
	\top: \mathcal{OF}^{\mathcal{C}} \otimes \mathcal{OF}^{\mathcal{C}} \to \mathcal{OF}^{\mathcal{C}}
$$ 
as a sum over vertices in the forest appearing in the second argument:
\begin{align}
\label{eq:natgrowthMKW}
	\omega_1 \top \omega_2 
	= \frac{1}{|\omega_2|} \sum_{v \in N(\omega_2)} \omega_1 \top_{\!v}\ \omega_2,
\end{align}
where $\omega_1 \top_{\!v}\ \omega_2$ is defined by grafting all roots in $\omega_1$ onto the vertex $v$, and $|\omega_2|$ is the number of vertices in $\omega_2$. The outgoing edges from the vertex $v$ are shuffled with the edges added by the grafting. For example:
\begin{align*}
	3\ \Forest{[a[b]]}\Forest{[c]} \top\ \Forest{[d[e][f]]}
	&= \Forest{[d[a[b]][c][e][f]]}+\Forest{[d[a[b]][e][c][f]]}+\Forest{[d[a[b]][e][f][c]]}+\Forest{[d[e][a[b]][c][f]]}\\
	&\quad +\Forest{[d[e][a[b]][f][c]]}+\Forest{[d[e][f][a[b]][c]]}+\Forest{[d[e[a[b]][c]][f]]}+\Forest{[d[e][f[a[b]][c]]]}.
\end{align*}

\begin{proposition} 
\label{prop::MKWNaturalGrowth}
The bilinear map $\top$ is a natural growth operator for the MKW Hopf algebra.
\end{proposition}

\begin{proof}
We have to show that the reduced coproduct and the map $\top$ satisfy
\begin{align*}
	\hat{\Delta}_{\scriptscriptstyle{\text{MKW}}}(x \top y) 
	= x \otimes y + x^{(1)} \otimes (x^{(2)} \top y ),
\end{align*}
for a primitive element $y$ in $\mathcal{H}_{\scriptscriptstyle{\text{MKW}}}$. We will first show the identity
\begin{align*}
	\hat{\Delta}_{\scriptscriptstyle{\text{MKW}}}(x \top_{\!v\ } y) 
	=  x \otimes y + x^{(1)} \otimes (x^{(2)} \top_{\!v}\ y ),
\end{align*}
for each vertex $v$ in $y$, and then obtain the result by summing over $v$ on both sides. To show this, we partition the left-admissible cuts of the tree $x \top_{\!v} y$ into two sets:
\begin{enumerate}

\item Let $c$ be a left admissible cut of $x \top_{\!v}\  y$ that cuts below the vertex $v$, then no edges of $x$ can be cut. Hence $P^c(x \top_{\!v} y)=x \top_{\!v}\  P^c(y)$ and $T^c(x \top_{\!v}\  y)=T^c(y)$. However, since $y$ is assumed to be primitive, the sum over these cuts will be zero.

\item Let $c$ be a cut of $x \top_{\!v}\  y$ that cuts above the vertex $v$, then we claim that $P^c(x \top_{\!v}\  y)=P^c(x)\shuffle P^c(y)$ and $T^c(x \top_{\!v}\  y)=T^c(x)\top_{\!v}\  T^c(y)$, where we understand $P^c(x)$ as being the restriction of the cut $c$ to the forest $x$ and similarly for the other terms. Then, since $y$ is primitive, the sum over these cuts will contribute $x \otimes y + x^{(1)} \otimes (x^{(2)} \top_{\!v}\  y )$ to $\hat{\Delta}_{\scriptscriptstyle{\text{MKW}}}(x \top_{\!v}\  y)$. To see that $P^c(x \top_{\!v} y)=P^c(x)\shuffle P^c(y)$, recall that the roots of $x$ are shuffled with the branches outgoing from $v$ in the grafting $x \top_{\!v} y$. Hence each planar order of the cut branches in $P^c(x \top_{\!v}\  y)$ will be recovered in the shuffle product $P^c(x)\shuffle P^c(y)$, and each order in the shuffle product will appear in a cut of $x \top_{\!v}\  y$. Branches cut off from other vertices than $v$ will appear in the shuffle product on both sides of the equality. The equality $T^c(x \top_{\!v} y)=T^c(x)\top_{\!v} T^c(y)$ says that cutting edges off $x$ before you graft is the same as cutting edges off $x$ after you graft, which clearly holds.

\end{enumerate}
The result is obtained by summing over the two cases.
\end{proof}

\begin{remark}
\label{rmk:BCK}
The natural growth operator on the BCK Hopf algebra, which was used in the computation of primitive elements in Example \ref{ex:piBCK} of Section \ref{sect:Foissy}, is defined on non-planar rooted forests similarly to \eqref{eq:natgrowthMKW}, by omitting the shuffling of the outgoing edges from the vertex $v$ with the edges added by the grafting in the prescription of the operation $\top_{\!v\ }$.
\end{remark}

Since we have found a natural growth operator on $\mathcal{H}_{\scriptscriptstyle{\text{MKW}}}$, all results from Section \ref{sect:Foissy} can be applied. Indeed, we can find the primitive elements by using the projection map $\pi$ defined recursively  by:
\begin{align*}
	\pi(x)=x-x^{(1)} \top \pi(x^{(2)}).
\end{align*}
For example:
\begin{align*}
	\pi(\Forest{[]}\Forest{[[]]})
	&=\Forest{[]}\Forest{[[]]}-( \Forest{[]}\top \pi(\Forest{[[]]})
		+\Forest{[]}\top \pi(\Forest{[]}\Forest{[]})+2\Forest{[]}\Forest{[]}\top \pi(\Forest{[]})  )\\
	&=\Forest{[]}\Forest{[[]]}- \Forest{[]}\top (\Forest{[[]]}- \Forest{[]}\top \Forest{[]} )
		-\Forest{[]}\top (\Forest{[]}\Forest{[]} - \Forest{[]}\top \Forest{[]} )-2\Forest{[[][]]}\\
	&=\Forest{[]}\Forest{[[]]}-\Forest{[]}\top 0 - \frac{1}{2}\Forest{[[]]}\Forest{[]}
		-\frac{1}{2}\Forest{[]}\Forest{[[]]}+\Forest{[[][]]}+\frac{1}{2}\Forest{[[[]]]}-2\Forest{[[][]]}\\
	&=\frac{1}{2}[\Forest{[]},\Forest{[[]]}]+\frac{1}{2}\Forest{[[[]]]}-\Forest{[[][]]},
\end{align*}
which is indeed primitive for the coproduct $\Delta_{\scriptscriptstyle{\text{MKW}}}$. We also see that any tree, except the single-vertex tree, is projected to zero, because any non-trivial tree $\tau$ can be written as $\tau=\omega \top \Forest{[]}$, for some forest $\omega$. The property \eqref{eq::NaturalGrowth} then yields:
\begin{align*}
	\pi(\omega \top \Forest{[]})
	&=\omega \top \Forest{[]} 
	- \big(\omega \top \Forest{[]} 
	+ \omega^{(1)} \top \pi(\omega^{(2)}\top \Forest{[]} ) \big),
\end{align*}
and the statement follows by induction over the degree of the forest. \\

Next, recall Proposition \ref{prop::FreeAssociative}. It states that the graded dual to $\mathcal{H}_{\scriptscriptstyle{\text{MKW}}}$, i.e., the planar Grossman--Larson Hopf algebra $\mathcal{H}_{\scriptscriptstyle{\text{GL}}}$, is a free associative algebra. In a later section, we will use this result to show that planarly branched rough paths can be seen as geometric rough paths, similarly to what was done for branched rough paths in  \cite{BoedihardjoChevyrev2019}. See also \cite{HairerKelly2012}. As a preliminary step, we show an isomorphism between $\mathcal{H}_{\scriptscriptstyle{\text{GL}}}$ and a tensor Hopf algebra. Let $T(\mathcal{PT}^{\mathcal{C}})=(\mathcal{OF}^{\mathcal{C}},\cdot,\Delta_{\shuffle})$ denote the tensor Hopf algebra over planar rooted trees, which consists of ordered forests together with concatenation and deshuffle. We now construct an Hopf algebra isomorphism between $\mathcal{H}_{\scriptscriptstyle{\text{GL}}}$ and $T(\mathcal{PT}^{\mathcal{C}})$. In the recent work \cite{alkaabi2022algebraic} it was identified with Gavrilov's so-called K-map \cite{Gavrilov2006}. See also \cite{Foissy2018post}.    

\begin{proposition} \cite{alkaabi2022algebraic,Foissy2018post}
\label{Prop:HopfAlgebraMorphism}
The map $\varphi: \mathcal{H}_{\scriptscriptstyle{\text{GL}}} \to T(\mathcal{PT}^{\mathcal{C}})$, defined for $\tau \in \mathcal{PT}$ 
by $\varphi(\tau):=\tau$ and extended by $\varphi(\omega_1 \ast \omega_2):=\varphi(\omega_1) \cdot \varphi(\omega_2)$ is a Hopf algebra isomorphism.
\end{proposition}

\begin{proof}
For the reader's benefit, we provide a detailed argument by checking all the properties:
\begin{enumerate}
\item $\varphi$ is injective: Following \eqref{GLprodinv2}, any forest $\tau_1 \cdots \tau_n$ can be written as a difference between a Grossman--Larson product $\ast$ and a linear combination of forests consisting of fewer trees.
\begin{align*}
	\tau_1 \cdots \tau_n 
	&=\tau_1 * (\tau_2 \cdots \tau_{n}) - \sum_{j=2}^n\tau_2 \cdots (\tau_1 \graft \tau_j) \cdots \tau_{n}.
\end{align*}
Applying this rewriting of forests in terms of Grossman--Larson products inductively, we see that we can express any forest as a Grossman--Larson polynomial, where the unique term with $n-1$ products is given by the product of the trees in the forest:
\begin{align*}
	\tau_1 \cdots \tau_n 
	&= \tau_1 \ast \cdots \ast \tau_n - \tau'_1 \ast \cdots \ast \tau'_{n-1} + \cdots \pm \tau''.
\end{align*}
Hence we have that:
\begin{align*}
	\varphi(\tau_1 \cdots \tau_n)
	=\tau_1 \cdots \tau_n + \text{ shorter forests.}
\end{align*}
In particular, no forest maps to zero. Hence the kernel of $\varphi$ is trivial, and the map must be injective.
\item $\varphi$ is surjective: $\varphi$ is surjective on trees by definition. Let $\tau_1 \cdots \tau_n$ be an arbitrary forest of $n$ trees, then:
\begin{align*}
\tau_1 \cdots \tau_n 
	&= \varphi(\tau_1) \cdot (\tau_2 \cdots \tau_n) \\
	&=\varphi(\tau_1 \ast \varphi^{-1}(\tau_2 \cdots \tau_n)   ),
\end{align*}
and the property follows by induction over the number of trees.
\item $\varphi$ is an algebra morphism: This follows by definition.
\item $\varphi$ is a coalgebra morphism: It is clear that
\begin{align*}
	(\varphi \otimes \varphi)\Delta_{\shuffle}(\tau)
	=\Delta_{\shuffle}(\varphi(\tau)),
\end{align*}
whenever $\tau$ is a tree. Suppose the property holds for all forests of at most $n-1$ trees, then:
\begin{align*}
	\Delta_{\shuffle}(\varphi(\tau_1 \cdots \tau_n  )  )
	&= \Delta_{\shuffle}(\varphi(\tau_1 \ast \cdots \ast \tau_n - \omega'  )  ),
\end{align*}
where $\omega'$ is a shorter forest. Hence:
\begin{align*}
	\Delta_{\shuffle}(\varphi(\tau_1 \cdots \tau_n  )  )
	&= \Delta_{\shuffle}(\tau_1 \cdots \tau_n) - (\varphi \otimes \varphi)\Delta_{\shuffle}(\omega')\\
	&=\Delta_{\shuffle}(\tau_1) \cdots \Delta_{\shuffle}(\tau_n) - (\varphi \otimes \varphi)\Delta_{\shuffle}(\omega')\\
	&=(\varphi \otimes \varphi)(\Delta_{\shuffle}(\tau_1) \ast \cdots \ast \Delta_{\shuffle}(\tau_n) -\Delta_{\shuffle}(\omega') )\\
	&=(\varphi \otimes \varphi)\Delta_{\shuffle}( \tau_1 \ast \cdots \ast \tau_n - \omega' )\\
	&=(\varphi \otimes \varphi)\Delta_{\shuffle}( \tau_1 \cdots \tau_n),
\end{align*}
and the property follows by mathematical induction.
\end{enumerate}
\end{proof}

We conclude this section by remarking that all of Foissy's results presented in Section \ref{sect:Foissy} now also apply to the MKW Hopf algebra, such as the classification of finite-dimensional comodules and the classification of bialgebra endomorphisms.

%%%%%%%%%%%%%%%%%%%%%%%%%%%%%%%%%%%
%%%%%%%%%%%%%%%%%%%%%%%%%%%%%%%%%%%

\section{Rough paths in a combinatorial Hopf algebras}
\label{sec:roughpaths}

The BCK Hopf algebra has played an important role in the study of rough differential equations on Eucliean spaces, in that it describes Gubinelli's \cite{Gubinelli2003} branched rough paths. This was generalized to planarly branched rough paths in \cite{CurryEbrahimiFardManchonMuntheKaas2018}, for the study of rough differential equations on homogeneous spaces. In this generalization, the role of the BCK Hopf algebra is now instead played by the MKW Hopf algebra. Hence, in our aim to put the MKW Hopf algebra on the same level as the BCK Hopf algebra, we want to generalize branched rough path results to planarly branched rough path results. One such result, by Boedihardjo and Chevyrev \cite{BoedihardjoChevyrev2019}, is that branched rough paths can be seen as geometric rough paths over a larger index set. This follows from Foissy's isomorphism between the BCK Hopf algebra and a shuffle Hopf algebra. Our aim is now to use the Hopf algebra isomorphism, $\varphi: \mathcal{H}_{\scriptscriptstyle{\text{GL}}} \to T(\mathcal{PT}^{\mathcal{C}})$, from the previous section to show how planarly branched rough paths can be seen as geometric rough paths. In order to address this topic, we first recall some basic rough path theory. We furthermore note some interesting consequences that arise when one considers the cointeraction $\rho_{\graft}$ from Section \ref{sec:GOconstruction} in this setting. We consider rough paths in the framework of combinatorial Hopf algebras as presented in \cite{CurryEbrahimiFardManchonMuntheKaas2018}\footnote{Generally speaking, there is no consensus on what a good definition of combinatorial Hopf algebra should be. Regarding our aim, we shall be using the following one.}. See also \cite{TZ2020}.

\begin{definition}
\label{def:combHopfAlg}
A combinatorial Hopf algebra $(\mathcal{H},\odot,\Delta,\eta,\epsilon)$ consists of a graded connected Hopf algebra $\mathcal{H} = \oplus_{n=0}^{\infty}\mathcal{H}_n$ over a field $\mathbb{K}$ of characteristic zero, together with a basis $\mathcal{B}=\cup_{n \geq 0} \mathcal{B}_n$ of homogeneous elements, such that:
\begin{enumerate}
\item There exist two positive constants $B$ and $C$ such that the dimension of $\mathcal{H}_n$ is bounded by $BC^n$.
\item The structure constants $c_{x y}^{z}$ and $c_{z}^{x y}$ of the product respectively the coproduct, defined for all elements $x,y,z \in \mathcal{B}$ by
\begin{align*}
	x \odot y 
	&= \sum_{z \in \mathcal{B}} c_{x y}^{z} z, \\
	\Delta(z) 
	&= \sum_{x,y \in \mathcal{B}} c_{z}^{x y} x \otimes y,
\end{align*}
are non-negative integers.
\end{enumerate}
We furthermore say that $\mathcal{H}$ is non-degenerate if $\mathcal{B} \cap \text{Prim}(\mathcal{H})=\mathcal{B}_1$.
\end{definition}

\noindent Examples of combinatorial Hopf algebras include, for instance, the shuffle Hopf algebra, the Butcher--Connes--Kreimer Hopf algebra as well as the Munthe-Kaas--Wright Hopf algebra. 

In this work, by a rough path, we mean the following.

\begin{definition} 
\label{def::RoughPath}
Let $\mathcal{H}=\oplus_{n\geq 0} \mathcal{H}_n$ be a non-degenerate combinatorial Hopf algebra with basis $\mathcal{B}$ (in the sense of Definition \ref{def:combHopfAlg}) and unit $1$. For $\gamma \in (0,1]$, a $\gamma$-regular $\mathcal{H}$-rough path is a two-parameter family $\mathbb{X}=(\mathbb{X}_{st})_{s,t \in \mathbb{R}}$ of linear forms on $\mathcal{H}$ such that $\langle \mathbb{X}_{st},1\rangle =1$ and: 
\begin{enumerate}
\item For any $s,t \in \mathbb{R}$ and any $x,y \in \mathcal{H}$, the following identity holds
\begin{align*}
	\langle \mathbb{X}_{st},x \odot y \rangle 
	= \langle \mathbb{X}_{st},x \rangle\langle \mathbb{X}_{st}, y \rangle.
\end{align*}
\item For any $s,t,u \in \mathbb{R}$, Chen's identity holds
\allowdisplaybreaks
\begin{align*}
	\mathbb{X}_{su} \ast \mathbb{X}_{ut}=\mathbb{X}_{st},
\end{align*}
where $\ast$ is the convolution product for linear forms on $\mathcal{H}$, defined in terms of the coproduct on $\mathcal{H}$.
\item For any $n \geq 0$ and any $x \in \mathcal{B}_n$, we have estimates
\begin{align*}
	\sup_{s \neq t} \frac{|\langle \mathbb{X}_{st},x \rangle |}{|t-s|^{\gamma |x|}} < \infty,
\end{align*}
where $|x|=n$ denotes the degree of the element $x \in \mathcal{B}_n$.
\end{enumerate}
\end{definition}

A geometric rough path lives over the tensor Hopf algebra with shuffle product, whereas a branched rough path is defined over the BCK Hopf algebra. We refer the reader to \cite{FH2020,FV2010,HairerKelly2012} for details on the theory of rough paths. Planarly branched rough paths \cite{CurryEbrahimiFardManchonMuntheKaas2018} are defined over the MKW Hopf algebra. 

We shall also need the notion of a truncated rough path.

\begin{definition} \label{def::TruncatedRoughPath}
Let $\mathcal{H}^{(N)}=\oplus_{0 \leq n \leq N} \mathcal{H}_n$ and let $\gamma \in (0,1]$. A $\gamma$-regular $N$-truncated $\mathcal{H}$-rough path is a two-parameter family $\mathbb{X}=(\mathbb{X}_{st})_{s,t \in \mathbb{R}}$ of linear forms on $\mathcal{H}^{(N)}$ such that $\langle \mathbb{X}_{st},1\rangle =1$ and: 
\begin{enumerate}
\item For any $s,t \in [0,1]$ and any $x,y \in \mathcal{H}$, with $|x|+|y| \leq N$, the following identity holds
\begin{align*}
	\langle \mathbb{X}_{st},x \odot y \rangle 
	= \langle \mathbb{X}_{st},x \rangle\langle \mathbb{X}_{st}, y \rangle.
\end{align*}
\item For any $s,t,u \in [0,1]$, Chen's identity holds
\allowdisplaybreaks
\begin{align*}
	\mathbb{X}_{su} \ast \mathbb{X}_{ut}=\mathbb{X}_{st},
\end{align*}
where $\ast$ is the convolution product restricted to $\mathcal{H}^{(N)}$.
\item For any $n \geq 0$ and any element $x \in \mathcal{B}_n$ of degree $|x|=n \leq N$, we have estimates
\begin{align*}
	\sup_{s \neq t} \frac{|\langle \mathbb{X}_{st},x \rangle |}{|t-s|^{\gamma |x|}} < \infty.
\end{align*}
\end{enumerate}
\end{definition}

\begin{theorem}[\cite{CurryEbrahimiFardManchonMuntheKaas2018}]
Let $\gamma \in (0,1]$ and let $N= \lfloor \frac{1}{\gamma} \rfloor$. Any $\gamma$-regular $N$-truncated $\mathcal{H}$-rough path admits a unique extension to a $\gamma$-regular $\mathcal{H}$-rough path.
\end{theorem}

\begin{definition}
\label{def::Signature}
Let $\mathbb{X}: [0,1]^{\otimes 2} \otimes \mathcal{H}^{(N)} \to \mathbb{R}$ be a $\gamma$-regular $N$-truncated $\mathcal{H}$-rough path and let $\tilde{\mathbb{X}}_{st}$ denote its unique extension. The signature of $\mathbb{X}$ is defined as:
\begin{align*}
	\mathrm{Sig}(\mathbb{X})=\tilde{\mathbb{X}}_{01}.
\end{align*}
\end{definition}

We shall also need the notion of translations of rough paths, from \cite{Rahm2021RP}. Recall that the definition of a combinatorial Hopf algebra comes with a choice of a basis $\mathcal{B}$. This basis will also play a role with respect to the concept of translations of rough paths. Let $(\mathcal{H},\odot,\Delta,\eta,\epsilon)$ be a non-degenerate combinatorial Hopf algebra with basis $\mathcal{B}$ and $\mathcal{B}_1=\{e_0,\ldots,e_{m-1}\}$. Let $(\mathcal{H}^{\ast},\ast,\Delta_{\odot})$ denote the graded dual space, with convolution product $\ast$ dual to $\Delta$ and coproduct $\Delta_{\odot}$ dual to the product $\odot$. Let $(\overline{\mathcal{H}^{\ast}},\ast,\Delta_{\odot})$ be the completion of $\mathcal{H}^{\ast}$ with respect to the grading, together with the continuously extended product $\ast$, and $\Delta_{\odot}$ extended continuously to a map $\overline{\mathcal{H}^{\ast}} \to \mathcal{H}^{\ast}\overline{\otimes} \mathcal{H}^{\ast}$. 

A translation is a family of maps $T_\mathbf{v}$, parametrised by the $m$ primitive elements $\mathbf{v}=(v_0,\ldots,v_{m-1} ) \in P(\overline{\mathcal{H}^{\ast}})^{\times m}$ satisfying the following.

\begin{definition} 
\label{def::Translation}
A family of algebra morphisms $T_{\bf{v}}: \overline{\mathcal{H}^{\ast}} \to \overline{\mathcal{H}^{\ast}}$ is a translation if
\begin{enumerate}
\item $T_{\bf{v}}(e_i)=e_i+v_i$ for every $e_i \in \mathcal{B}_1$ and some ${\bf{v}}=(v_0,\ldots,v_{m-1})$, $v_i \in \overline{\mathcal{H}^{\ast}}$ primitive.

\item $T_{\bf{v}} \circ T_{\bf{u}} = T_{\bf{v} + T_{\bf{v}}({\bf{u}})}$, where $T_{\bf{v}}({\bf{u}})=(T_{\bf{v}}(u_0), \ldots, T_{\bf{v}}(u_{m-1}))$.

\item For each $\mathcal{H}$-rough path $\mathbb{X}_{st}$, the pointwise translation $T_{\bf{v}}(\mathbb{X}_{st})$ is a $\mathcal{H}$-rough path:
\begin{enumerate}

\item $T_{\bf{v}}$ maps characters to characters.

\item $T_{\bf{v}}$ is a morphism with respect to the convolution product of $\overline{\mathcal{H}^*}$.

\item The bound 
\begin{align*}
	\sup_{s \neq t} \frac{|\langle T_{\bf{v}}(\mathbb{X}_{st}),x \rangle |}{|t-s|^{\gamma |x|}} < \infty
\end{align*}
holds.
\end{enumerate}

\item \label{item:coaction} There exists a coaction $\rho_T : \mathcal{H} \to \mathcal{S}(P(\mathcal{H}^{\ast}) \times \mathcal{B}_1)\otimes \mathcal{H}$ such that $\langle T_{\bf{v}}(\chi),x\rangle = \langle e^{\bf{v}} \otimes \chi,\rho_T(x)\rangle$, where $e^{\mathbf{v}}=\exp( \sum_i (v_i,e_i) )$ in the free commutative co-commutative Hopf algebra $\mathcal{S}(P(\mathcal{H}^{\ast}) \times \mathcal{B}_1)$.
\end{enumerate}
\end{definition}

Denote the free commutative co-commutative Hopf algebra by $(\mathcal{S}(P(\mathcal{H}^{\ast}) \times \mathcal{B}_1),\centerdot,\Delta_{\centerdot})$. Then we have the following result.

\begin{theorem}[\cite{Rahm2021RP}]
Let $T_{\mathbf{v}}: \overline{\mathcal{H}^{\ast}} \to \overline{\mathcal{H}^*}$ and $\rho_T : \mathcal{H} \to \mathcal{S}(P(\mathcal{H}^{\ast}) \times \mathcal{B}_1)\otimes \mathcal{H}$ be related by:
\begin{align*}
\langle T_{\bf{v}}(\chi),x\rangle 
= \langle e^{\bf{v}} \otimes \chi,\rho_T(x)\rangle. 
\end{align*}
Then $T_{\mathbf{v}}$ is a translation if and only if the following holds:
\begin{enumerate}
\item $\rho_T$ is a cointeraction.
\item The following two identities hold:
\begin{align}
	( \mathrm{id} \otimes \rho_T)\rho_T 
	&= m^{1,2}(( \mathrm{id} \otimes \rho_T \otimes \mathrm{id})(\Delta_{\centerdot} \otimes \mathrm{id})\rho_T  ),\label{eq::cotranslation1} \\
	\langle v_i + e_i,x\rangle
	&=\langle e^{\mathbf{v}} \otimes e_i,\rho_T(x) \rangle, \label{eq::cotranslation2}
\end{align}
where $\rho_T$ is extended to $\mathcal{S}(P(\mathcal{H}^{\ast})\times \mathcal{B}_1)$ by letting it act on the first component.
\end{enumerate}
\end{theorem}

The above definition appeared in the context of solving differential equations using rough paths. There, the set $\mathcal{B}_1$ corresponds to the vector fields in the differential equation. Hence, the characters over $\mathcal{S}(P(\mathcal{H}^{\ast})\times \mathcal{B}_1)$ are in bijection with all possible directions to translate the vector fields in. The main results from \cite{BrunedChevyrevFrizPreiss} and \cite{Rahm2021RP} say then that translating vector fields before solving the equation is equivalent to translating the solution of the unmodified equation. We will now show that the coaction $\rho_{\graft}$ defined in Section \ref{sec:GOconstruction} satisfies all of these axioms, but relative to a different cointeracting Hopf algebra. In particular, the interpretation in terms of translating vector fields is not applicable.\\

Consider the MKW algebra, a concrete translation map was constructed in \cite{Rahm2021RP}. One starts by defining $T_{\mathbf{v}}(e_i)=e_i+v_i$, then extending this to all of $Lie(\mathcal{PT}^{\mathcal{C}})$ as a post-Lie morphism. Finally one extends the translation map to all ordered forests as a morphism for the planar Grossman--Larson product $\ast$ defined earlier. One can then show that this map is also a morphism for the concatenation product. Hence, it is a translation map for both $\mathcal{H}_{\scriptscriptstyle{\text{MKW}}}$ and $T(\mathcal{PT}^{\mathcal{C}})^{\ast}$. For the rest of this section, the notation $T_{\mathbf{v}}$ refers to this concrete translation map, which we shall call post-Lie translation of planarly branched rough paths. The dual map, $\rho_T$, can be described by contracting admissible subtrees. See \cite{Rahm2021RP} for details.\\

The elements in $\mathcal{B}_1$ should be thought of as vector fields in some differential equation, then there is a bijection between the characters of $\mathcal{S}(P(\mathcal{H}^{\ast})\times \mathcal{B}_1)$ and the set of possible directions to translate these vector fields in, where translation is understood in the way of $T_{\mathbf{v}}(e_i)=e_i+v_i$. Then each character of $\mathcal{S}(Lie(\mathcal{PT}^{\mathcal{C}})\times \mathcal{B}_1)$ gives an automorphism on $\mathcal{G}_{\odot}$, such that the dual map satisfies identity \eqref{eq::cotranslation1}. Similarly, the coaction $\rho_{\graft}$ assigns to each character over $T(\mathcal{PT}^{\mathcal{C}})$, an automorphism on $\mathcal{G}_{\odot}$, where again identity \eqref{eq::CoTranslation} is satisfied. Note that \eqref{eq::CoTranslation} and \eqref{eq::cotranslation1} differ only by the specific product and coproduct in their respective Hopf algebras. Hence the $T_{\mathbf{v}}$'s generate a subgroup of the automorphism group, and so does $\rho_{\graft}$. One can now ask to what extent these subgroups overlap. The following proposition states that these subgroups intersect only in the identity operator.

\begin{proposition} \label{prop::DisjointCointeractions}
Let $e^{\mathbf{v}}$ be a character over $\mathcal{S}(Lie(\mathcal{PT}^{\mathcal{C}}) \times \mathcal{C}  )$ and let $\xi$ be a character over $T(\mathcal{PT}^{\mathcal{C}})$. Suppose that
\begin{align*}
	\langle e^{\mathbf{v}}\otimes A,\rho_T(x) \rangle
	=\langle \xi \otimes A,\rho_{\graft}(x)\rangle
\end{align*}
for all $A \in \mathcal{G}_{\odot}$ and all $x \in \mathcal{OF}^{\mathcal{C}}$. Then $e^{\mathbf{v}}=\xi=1$.
\end{proposition}

\begin{proof}
Note that $T_{\mathbf{v}}(\Forest{[i]})=\Forest{[i]}+v_i$ and that $\xi \graft \Forest{[i]}= \Forest{[i[\xi]]}$, where a vertex is decorated by $\xi$ to represent that $\xi$ is grafted there. Then, since $T_{\mathbf{v}}$ and $\xi \graft$ are equal, we must have that $v_i=B_+^{i}(\xi - 1 )$. But then:
\begin{align*}
	\xi \graft \Forest{[i[j]]}
	&= \Forest{[i[\xi][j[\xi]]]},\\
	T_{\mathbf{v}}(\Forest{[i[j]]})
	&=\Forest{[i[j[\xi]][\xi]]}+\Forest{[i[\xi[j[\xi]]]]}.
\end{align*}
These coincide only if $\xi=1$.
\end{proof}

\begin{remark}
There is also the concept of substitutions of rough paths, defined by replacing the first two axioms in Definition \ref{def::Translation} by:
\begin{enumerate}
\item $\mathrm{S}_{\mathbf{v}}(e_i)=v_i$.
\item $\mathrm{S}_{\mathbf{v}}\circ \mathrm{S}_{\mathbf{u}}=\mathrm{S}_{\mathrm{S}_{\mathbf{v}}(\mathbf{u})}$.
\end{enumerate}
This is equivalent to translations up to changing the vector $\mathbf{v}$. The identity \eqref{eq::cotranslation1} for the dual coaction is then replaced by:
\begin{align*}
	( \mathrm{id} \otimes \rho_\mathrm{S})\rho_\mathrm{S}
		=(\rho_\mathrm{S} \otimes \mathrm{id})\rho_\mathrm{S},
\end{align*}
where $\rho_\mathrm{S}$ is extended to $\mathcal{S}(Lie(\mathcal{PT}^{\mathcal{C}}) \times \mathcal{C}  )$ in the natural way. This should be compared with identity \eqref{eq::CoSubstitution}. Hence $\mathcal{G}_{\odot}$ acting on $\mathcal{G}_{\odot}$ via $\graft$ looks like a translation, and $\mathcal{G}_{\ast}$ acting on $\mathcal{G}_{\odot}$ via $\graft$ looks like a substitution.
\end{remark}

We conclude this section by giving a combinatorial description of the coaction $\rho_{\graft}$. Recall that $\rho_{\graft}$ is dual to the product $\graft$, which is given by left grafting a forest onto another forest. It is related to the Grossman--Larson product by $A \ast B =A_{(1)}(A_{(2)}\graft B)$. Hence, one can think of $\graft$ as being the part of $\ast$ where all of $A$ is grafted and nothing is concatenated. Dually, this means that $\rho_{\graft}$ is the part of $\Delta_{\scriptscriptstyle{\text{MKW}}}$ where nothing is deconcatenated. Let $\tau$ be a planar tree, then
\begin{align*}
	\rho_{\graft}(\tau)=\Delta_{\scriptscriptstyle{\text{MKW}}}(\tau).
\end{align*}
Let $\omega_1,\omega_2$ be two forests, then:
\begin{align*}
	\rho_{\graft}(\omega_1 \cdot \omega_2)
	=\rho_{\graft}(\omega_1)(\shuffle \otimes \cdot  )\rho_{\graft}(\omega_2),
\end{align*}
meaning that one performs left-admissible cuts on each tree individually and shuffle the pruned parts on the left side of the tensor product. For example:
\begin{align*}
	\rho_{\graft}(\Forest{[a[b][c]]}\Forest{[d[e]]}\Forest{[f[g]]}  )
	&=1 \otimes \Forest{[a[b][c]]}\Forest{[d[e]]}\Forest{[f[g]]} 
		+ \Forest{[b]} \otimes \Forest{[a[c]]}\Forest{[d[e]]}\Forest{[f[g]]}
		+\Forest{[e]} \otimes \Forest{[a[b][c]]}\Forest{[d]}\Forest{[f[g]]}
		+\Forest{[g]} \otimes \Forest{[a[b][c]]}\Forest{[d[e]]}\Forest{[f]}
		+\Forest{[b]}\Forest{[c]} \otimes \Forest{[a]}\Forest{[d[e]]}\Forest{[f[g]]}\\
	&+\Forest{[b]}\shuffle \Forest{[e]} \otimes \Forest{[a[c]]}\Forest{[d]}\Forest{[f[g]]}
		+\Forest{[b]}\shuffle \Forest{[g]} \otimes \Forest{[a[c]]}\Forest{[d[e]]}\Forest{[f]}
		+\Forest{[e]}\shuffle \Forest{[g]} \otimes \Forest{[a[b][c]]}\Forest{[d]}\Forest{[f]}
		+(\Forest{[b]}\Forest{[c]})\shuffle \Forest{[e]}\otimes \Forest{[a]}\Forest{[d]}\Forest{[f[g]]}\\
	&+(\Forest{[b]}\Forest{[c]})\shuffle \Forest{[g]}\otimes \Forest{[a]}\Forest{[d[e]]}\Forest{[f]}
		+\Forest{[b]}\shuffle \Forest{[e]}\shuffle \Forest{[g]} \otimes \Forest{[a[c]]}\Forest{[d]}\Forest{[f]}
		+(\Forest{[b]}\Forest{[c]})\shuffle \Forest{[e]}\shuffle \Forest{[g]}\otimes \Forest{[a]}\Forest{[d]}\Forest{[f]}.
\end{align*}

%%%%%%%%%%%%%%%%%%%%%%%%%%%%%%%%%%%
%%%%%%%%%%%%%%%%%%%%%%%%%%%%%%%%%%%

\section{Planarly branched rough paths are geometric}
\label{sec:planarlyBRP}

Recall Definition \ref{def::Signature}, where the signature of a rough path appears. The problem of finding the kernel of the signature map was considered in \cite{BoedihardjoGengLyonsYang2014}, where it was shown that a geometric rough path has trivial signature if and only if the path is a so-called tree-like path. This was then generalised in \cite{BoedihardjoChevyrev2019}, where the authors used the Hopf algebra isomorphism between branched rough paths and geometric rough paths to classify the kernel of the signature map also for branched rough paths. We now aim to generalize this further to planarly branched rough paths by using the same method. We first recall the notion of geometric $\Pi$-rough paths from \cite{BoedihardjoChevyrev2019}, which generalises geometric rough paths to tensor algebras where different letters can have different weights. \\

Following \cite{BoedihardjoChevyrev2019}, we decompose the finite dimensional normed vector space $V$ into:
\begin{align*}
	V = V^1 \oplus \cdots \oplus V^k,
\end{align*}
to represent the different degrees (regularities) of the elements. The set of multi-indices is defined:
\begin{align*}
	\mathcal{A}_k := \{(r_1,\ldots,r_m): r_i \in \{1,\ldots,k\},m\geq 0 \}.
\end{align*}
Then we can write the tensor algebra over $V$ as:
\begin{align*}
	T(V)=\prod_{R=(r_1,\ldots,r_m)\in \mathcal{A}_k} V^{\otimes R},
\end{align*}
where $V^{\otimes R} = V^{r_1} \otimes \cdots \otimes V^{r_m}$. Let $\Pi=(p_1,\ldots,p_k)$ be a so-called scaling tuple, with $p_1 \geq \cdots \geq p_k \geq 1$. Let $\pi_R$ denote the projection onto $V^{\otimes R}$. For $R \in \mathcal{A}_k$, define $n_j(R)=|\{i : r_i=j \}|$, an inhomogeneous analogue of the tensor degree is defined:
\begin{align*}
	\deg_{\Pi}(R):=\sum_{j=1}^k \frac{n_j(R)}{p_j}.
\end{align*}
To express truncated rough paths, consider the set of multi-indexes with degree at most $s \ge 0$:
\begin{align*}
	\mathcal{A}_{k,s}^{\Pi}=\{R \in \mathcal{A}_k : \deg_{\Pi}(R) \leq s \},
\end{align*}
and the ideal:
\begin{align*}
	B_s^{\Pi}=\{v \in T(V):\pi_R(v)=0, \forall R\in \mathcal{A}_{k,s}^{\Pi}   \}.
\end{align*}
The truncated tensor algebra at degree $s$ is then the quotient algebra:
\begin{align*}
	T^{(\Pi,s)}(V)=T(V)/B_s^{\Pi}.
\end{align*}
Let $T^{(\Pi,s)}_{\shuffle}(V)$ denote the graded dual shuffle algebra. With a slight abuse of notation, we denote the degree of an element $x \in T^{(\Pi,s)}_{\shuffle}(V)$ by $\deg_{\Pi}(x)$.

\begin{definition}
A geometric $\Pi$-rough path is a map $\mathbb{X}: [0,1]^{\otimes 2} \to T^{(\Pi,1)}(V)$ such that:
\begin{enumerate}
\item For any $s,t \in [0,1]$ and any $x,y \in T^{(\Pi,1)}_{\shuffle}(V)$, the following identity holds
\begin{align*}
	\langle \mathbb{X}_{st},x \shuffle y \rangle 
	= \langle \mathbb{X}_{st},x \rangle\langle \mathbb{X}_{st}, y \rangle,
\end{align*}
whenever $\deg_{\Pi}(x)+\deg_{\Pi}(y)\leq 1$.
\item For any $s,t,u \in [0,1]$, Chen's identity holds
\begin{align*}
	\mathbb{X}_{su} \cdot \mathbb{X}_{ut}=\mathbb{X}_{st},
\end{align*}
where the product is the truncated tensor product in $T^{(\Pi,1)}(V)$.
\item For any $n \geq 0$ and any $x \in T^{(\Pi,1)}(V)$, we have estimates
\begin{align*}
	\sup_{s \neq t} \frac{|\langle \mathbb{X}_{st},x \rangle |}{|t-s|^{\deg_{\Pi}(x)}} < \infty.
\end{align*}
\end{enumerate}
\end{definition}

As for all other notions of truncated rough paths, also geometric $\Pi$-rough paths have an extension theorem \cite{FV2010}: each geometric $\Pi$-rough path $\mathbb{X}$ has an unique extension $\mathbb{X}^{ex}: [0,1]^{\otimes 2} \to T(V)$ that satisfies the usual rough path axioms \cite{FH2020}. Recall that we denoted the signature map $\mathrm{Sig}(\mathbb{X}) := \mathbb{X}^{ex}_{0,1}$. The problem of finding the kernel for the signature map was solved in \cite{BoedihardjoChevyrev2019}, as being exactly the tree-like paths.

\begin{theorem}[\cite{BoedihardjoChevyrev2019}]
Let $\mathbb{X}$ be a geometric $\Pi$-rough path, then $\mathrm{Sig}(\mathbb{X})=1$ if and only if $\mathbb{X}_{0,\cdot}$ is tree-like.
\end{theorem}

We recall the definition of tree-like paths.

\begin{definition}
An $\mathbb{R}$-tree is a geodesic metric space which contains no subset homeomorphic to a circle.
\end{definition}

\begin{definition}
For a topological space $\mathcal{S}$, a continuous path $X: [0,1] \to \mathcal{S}$ is called tree-like if and only if there exists an $\mathbb{R}$-tree $\mathfrak{T}$, a continuous function $\phi: [0,1] \to \mathfrak{T}$, and a map $\chi: \mathfrak{T} \to \mathcal{S}$ such that $\phi(0)=\phi(1)$ and $X=\chi \circ \phi$.
\end{definition}

\medskip

Let $\gamma \in (0,1]$ and let $N= \lfloor \frac{1}{\gamma} \rfloor$. We now want to show that $\gamma$-regular $N$-truncated planarly branched rough paths are geometric $\Pi$-rough paths. Let $\tau_1,\tau_2,\tau_3,\ldots$ be an ordering of all planar trees such that $|\tau_1| \leq |\tau_2| \leq \cdots$, and let $k$ be the largest integer such that $|\tau_k| \leq N$. Let $B^j=\text{span}(\tau_j)$ and consider the space:
\begin{align*}
	B_k = B^1 \oplus \cdots \oplus B^k.
\end{align*}
Define $p_j=N/|\tau_j|$, the tuple $\Pi=(p_1,\ldots,p_k)$, and consider the space $T^{(\Pi,1)}(B_k)$. Note that this space consists of the forests $\tau_{i_1}\cdots \tau_{i_m}$ for which
\begin{align*}
	\deg_{\Pi}(\tau_{i_1}\cdots \tau_{i_m}) 
	= \sum_{j=1}^m \frac{|\tau_{i_j}|}{N} \leq 1,
\end{align*}
which are exactly the forests with at most $N$ vertices, i.e., the $N$-truncated vector space of planar forests. Now recall the map $\varphi: \mathcal{OF} \to \mathcal{OF}$, given as the identity on trees and extended as an algebra morphism by sending the Grossman--Larson product $\ast$ to concatenation. We denote the restriction to the $N$-truncated spaces by $\varphi^{(N)}: \mathcal{OF}^{(N)} \to T^{(\Pi,1)}(B_k)$.

\begin{theorem}
Let $\mathbb{X}_{st}$ be a $\gamma$-regular $N$-truncated planarly branched rough path, then $\varphi^{(N)}(\mathbb{X}_{st})$ is a geometric $\Pi$-rough path. Furthermore:
\begin{align*}
	\mathrm{Sig}(\mathbb{X})=\varphi^{-1}(\mathrm{Sig}(\varphi^{(N)}(\mathbb{X}))).
\end{align*}
\end{theorem}

\begin{proof}
The shuffle identity follows from $\varphi$ being a coalgebra morphism. Chen's identity follows from the fact that $\varphi$ is an algebra morphism. The analytical bound follows from equivalence of norms on finite-dimensional vector spaces. Lastly, it is clear that
\begin{align*}
\mathbb{X}=\pi^N(\varphi^{-1}(\mathrm{Sig}(\varphi^{(N)}(\mathbb{X}) ) ) ),
\end{align*}
and then the last statement follows by the uniqueness of the extension.
\end{proof}

\begin{lemma}
Let $\mathbb{X}_{st}$ be a $\gamma$-regular $N$-truncated planarly branched rough path, then $\mathbb{X}_{st}$ is tree-like if and only if $\varphi^{(N)}(\mathbb{X}_{st})$ is tree-like.
\end{lemma}

\begin{corollary} \label{Cor::TrivialSignatureTreeLike}
Let $\mathbb{X}_{st}$ be a $\gamma$-regular $N$-truncated planarly branched rough path, then $\mathrm{Sig}(\mathbb{X})=1$ if and only if $\mathbb{X}_{st}$ is tree-like.
\end{corollary}

\begin{corollary}
Let $\mathbb{X}_{st},\mathbb{Y}_{st}$ be two $\gamma$-regular $N$-truncated planarly branched rough paths, then $\mathrm{Sig}(\mathbb{X})=\mathrm{Sig}(\mathbb{Y})$ if and only if $\mathbb{X}^{-1}_{st} \ast \mathbb{Y}_{st}$ is tree-like.
\end{corollary}

Now consider the planarly branched differential equation:
\begin{align}
	dY_{st} =\sum_{i=i}^n f_i(Y_{st})dX_t^i, 
	\quad 
	Y_{ss} =y_0, \label{eq::RoughPlanarEquation}
\end{align}
where $Y_{s\cdot}:\mathbb{R} \to \mathcal{M}$, $f_i : \mathcal{M} \to T\mathcal{M}$, $X_{\cdot}^i: \mathbb{R} \to \mathbb{R}$, for some homogeneous space $\mathcal{M}$. Such differential equations were considered in \cite{CurryEbrahimiFardManchonMuntheKaas2018}. Denoting the Lie algebra that acts on $\mathcal{M}$ by $\mathfrak{g}$, one can represent flows on $\mathcal{M}$ in the space $C^{\infty}(\mathcal{M},\mathcal{U}(\mathfrak{g}))$, where $\mathcal{U}(\mathfrak{g})$ is the universal enveloping algebra of $\mathfrak{g}$. This space is a $D$-algebra, and hence there is a $D$-algebra morphism $F_{\bf{f}}: \mathcal{OF}^{\{1,\ldots,n\}} \to C^{\infty}(\mathcal{M},\mathcal{U}(\mathfrak{g}))$. 

\begin{definition}[\cite{CurryEbrahimiFardManchonMuntheKaas2018}]
A formal solution to equation \eqref{eq::RoughPlanarEquation} is given by
\begin{align*}
	\mathbb{Y}_{st}=\sum_{x \in \mathcal{OF}^{\{1,\ldots,n \}}}\langle \mathbb{X}_{st},x\rangle F_{\textbf{f}}(x),
\end{align*}
where $\mathbb{X}_{st}$ is a rough path in $\mathcal{H}_{\scriptscriptstyle{\mathrm{MKW}}}$ such that
\begin{align*}
	\langle \mathbb{X}_{st},\Forest{[i]}\rangle = X_t^i - X_s^i.
\end{align*}
\end{definition}

The formal solutions are parametrised by the rough paths over $\mathcal{H}_{\scriptscriptstyle{\text{MKW}}}$. Denote by
\begin{align}
	d\mathbb{Y}_{st}
	=\sum_{i=1}^n f_i(\mathbb{Y}_{st})d\mathbb{X}_t, 
	\quad 
	\mathbb{Y}_{ss}
	=y_0, \label{eq::RoughPlanarEquationFixed}
\end{align}
the equation where we fixed the rough path extension $\mathbb{X}_{st}$ of $X_{st}$. \\

Next consider the geometric differential equation:
\begin{align*}
	d\mathbb{Y}_{st}
	=\sum_{j=1}^k f_{\tau_j}(\mathbb{Y}_{st})d\mathbb{X}_{st}, 
	\quad 
	\mathbb{Y}_{ss}=y_0,
\end{align*}
where $Y_{s\cdot}: \mathbb{R}\to T(\mathcal{PT}^{\{1,\ldots,n \} })$, $\mathbb{X}_{st}$ is a rough path over $T(\mathcal{PT}^{ \{ 1,\ldots,n \} })$ and $f_{\tau_j}$ denotes multiplication by $\tau_j$. Then the solution to this geometric equation is given by:
\begin{align*}
	\mathbb{Y}_{st}
	= \sum_m \sum_{\tau_{i_1} \cdots \tau_{i_m} \in T(\mathcal{PT}^{\{1,\ldots,n \} })} 
	\langle \mathbb{X}_{st},\tau_{i_1} \cdots \tau_{i_m} \rangle f_{i_1} \cdots f_{i_m}.
\end{align*}

\begin{proposition}
The solutions to the planarly branched equation
\begin{align*}
	d\mathbb{Y}_{st}
	=\sum_{i=1}^n f_i(\mathbb{Y}_{st})d\mathbb{X}_t, 
	\quad 
	\mathbb{Y}_{ss}=y_0,
\end{align*}
and the geometric equation
\begin{align*}
	d\overline{\mathbb{Y}}_{st}
	=\sum_{j=1}^k f_{\tau_j}(\overline{\mathbb{Y}}_{st})d\varphi(\mathbb{X}_{st}), 
	\quad 
	\overline{\mathbb{Y}}_{ss}=y_0,
\end{align*}
coincide.
\end{proposition}

\begin{proof}
Both solutions are sums over evaluating a character on the same vector space, where both characters agree on a generating set for their respective product.
\end{proof}

%%%%%%%%%%%%%%%%%%%%%%%%%%%%%%%%%%%
%%%%%%%%%%%%%%%%%%%%%%%%%%%%%%%%%%%

\section{Geometric embedding for planar regularity structures}
\label{sec:geomembedding}

In \cite{BrunedKatsetsiadis2023}, the authors showed a geometric embedding for regularity structures. Regularity structures may be seen as a generalization of branched rough paths, and the authors of \cite{BrunedKatsetsiadis2023} were attempting to generalize the Hopf algebra isomorphism between branched and geometric rough paths to this setting. Their embedding comes in the form of a Hopf algebra isomorphism between the dual of the Hopf algebra for recentering and a quotient algebra of a tensor Hopf algebra. Following the program of generalizing BCK results to MKW results, regularity structures were generalized to planar regularity structures in \cite{Rahm2022}. In this generalization, the Hopf algebra for recentering was extended from non-planar trees to planar trees, using a construction based on the previous works \cite{BrunedManchon2022} and \cite{BruKatPostLie}. In this section, we will first recall this construction. We will then extend both the results from \cite{BrunedKatsetsiadis2023}, and our geometric embedding of planarly branched rough paths, by showing a Hopf algebra isomorphism between the planar recentering Hopf algebra and a quotient of a tensor algebra. This follows our aim of generalizing all BCK results to MKW results, and pushes our geometric embedding theme to its limits. We are now working over planar trees where both the edges and vertices are decorated, with elements from $\mathbb{N}^d$. We use the notation $I_a(\omega)$ to denote the tree obtained by grafting the forest $\omega$ onto a zero-decorated root by using an edge decorate by $a$. Let $\tilde{\mathcal{V}}$ denote the space spanned by all $I_a(\tau)$ for $\tau$ a Lie polynomial. Then $\tilde{\mathcal{V}}$ is a Lie algebra with Lie bracket:
\begin{align}
	[I_a(\tau_1),I_b(\tau_2)]
	=\Forest{[ [\tau_1,edge label = {node[below=0.1cm,fill=white,scale=0.5]{$a$}}][\tau_2,edge label 
	= {node[below=0.1cm,fill=white,scale=0.5]{$b$}}] ]}
	-\Forest{[ [\tau_2,edge label 
	= {node[below=0.1cm,fill=white,scale=0.5]{$b$}}][\tau_1,edge label 
	= {node[below=0.1cm,fill=white,scale=0.5]{$a$}}] ]}, \label{eq::ForestAsTree}
\end{align}
where the edge below $\tau_1$ is decorated by $a$ and the edge below $\tau_2$ is decorated by $b$. We then turn $\tilde{\mathcal{V}}$ into a post-Lie algebra by defining
\begin{align*}
	I_a(\tau_1)\graft I_b(\tau_2)=I_b(\tau_1 \graft_a \tau_2),
\end{align*}
where $\graft_a$ means left grafting by an edge decorated by $a$. We now define the operators $\uparrow^{\ell}$ and $\uparrow^{\ell}_v$ on the space $\tilde{\mathcal{V}}$. First, let $\ell=(0,\ldots,0,1,0,\ldots,0) \in \mathbb{N}^d$ be a unit vector, then we define $\uparrow^{\ell}_v \tau$ to be the tree obtained by increasing the decoration of the vertex $v$ in $\tau$ by $\ell$. We also define:
\begin{align*}
	\uparrow^{\ell}\tau = \sum_v \uparrow^{\ell}_v \tau.
\end{align*}
Next, let $\ell \in \mathbb{N}^d$ be arbitrary, then we can write $\ell=\ell_1+\cdots+\ell_k$ for $\ell_i$, $1 \le i \le k$, being unit vectors, and we define:
\begin{align*}
	\uparrow^{\ell}_v \tau 
	&= \uparrow^{\ell_1}_v\circ \cdots \circ \uparrow^{\ell_k}_v \tau,\\
	\uparrow^{\ell} \tau 
	&= \uparrow^{\ell_1}\circ \cdots \circ \uparrow^{\ell_k} \tau.
\end{align*}
We similarly define $\uparrow^{-\ell}_v$ to decrease the decoration. This yields the so-called deformed grafting product by:
\begin{align*}
	I_a(\tau_1) \dgraft I_b(\tau_2)
	:=\sum_v \sum_{\ell}{n_v \choose \ell }I_b(\tau_1 \graft_{a-\ell}^v(\uparrow^{-\ell}_v \tau_2   )  ),
\end{align*}
where $n_v$ is the decoration of the vertex $v$. We interpret terms to be zero whenever a decoration would be outside of $\mathbb{N}^d$, and $\graft^v$ means left grafting onto the vertex $v$. We now extend the space $\tilde{\mathcal{V}}$ to the space $\mathcal{V}$ by introducing the units in $\mathbb{N}^d$ as basis vectors:
\begin{align*}
	\mathcal{V}=\tilde{\mathcal{V}} \oplus \text{span}( \{X^i : i=(0,\ldots,0,1,0,\ldots,0)  \} ).
\end{align*}
We now define the product $\dgraftt$ on $\mathcal{V}$ by:
\begin{align*}
	X^i \dgraftt y
	&=\uparrow^iy,\\
	y \dgraftt X^i
	&=0,\\
	X^i \dgraftt X^j 
	&= 0,\\
	y \dgraftt z
	&=y \dgraft z,
\end{align*}
for $y,z \in \tilde{\mathcal{V}}$. We furthermore define the Lie bracket $[\cdot,\cdot]_0$ by:
\begin{align*}
	[X^i,X^j]_0
	&=0,\\
	[y,z]_0
	&=[y,z],\\
	[y,X^i]_0
	&=\downarrow^iy,
\end{align*}
where $\downarrow^iy$ sums over the edges in $y$ that are adjacent to the root, and subtracts $i$ from one decoration per term, for example:
\begin{align*}
	[[I_a(\tau_1),I_b(\tau_2)  ],X^i]_0
	&= \downarrow^i [I_a(\tau_1),I_b(\tau_2)]\\
	&=[I_{a-i}(\tau_1),I_{b}(\tau_2)]+[I_a(\tau_1),I_{b-i}(\tau_2)].
\end{align*}

\begin{proposition}[\cite{Rahm2022}]
$(\mathcal{V},\dgraftt,[\cdot,\cdot]_0)$ is a post-Lie algebra.
\end{proposition}

The Hopf algebra for recentering in planar regularity structures is obtained by applying the Guin--Oudom construction to this post-Lie algebra. Let $\mathfrak{T}$ denote the space of all planar trees with decorated vertices and decorated edges, then it was shown in \cite{Rahm2022} that we can represent the universal enveloping algebra of $(\mathcal{V},[\cdot,\cdot]_0)$ by this space. We recall the proof from \cite{Rahm2022} as it describes the representation.

\begin{lemma}\cite{Rahm2022} 
\label{Lemma::UniversalEnvelopingAlgebraIsTrees}
The Lie enveloping algebra of $(\mathcal{V},[\cdot,\cdot]_0)$ can be represented by the space $\mathfrak{T}$.
\end{lemma}

\begin{proof}\cite{Rahm2022}
It is clear that the $X^i$'s commute under the associative product on $\mathcal{U}(\mathcal{V})$, as their commutator Lie brackets vanish. Let $X^m$ be a commutative polynomial of $X^i$'s whose indices add up to $m \in \mathbb{N}^d$ and represent this object with a single vertex decorated by $m$. It is furthermore clear by the identification \eqref{eq::ForestAsTree} that we can represent the associative product of two planted trees by merging the roots of the trees letting the branches of the left argument be to the left of the branches of the right argument. It remains to represent the associative product between single vertices and planted trees. Let $X^i\tau$ be represented by the planted tree $\tau$, except with the root decorated by $i$. Note that this implies that $\tau X^i$ is represented by $\tau X^i=X^i\tau - [X^i,\tau]_0 = X^i \tau + \downarrow^i \tau$. Since $\downarrow^i \tau$ is again a planted tree, this means that we can always write any element in $\mathcal{U}(\mathcal{V})$ as a sum of elements of the form $X^m \tau_1 \cdots \tau_k$, which we represent by the tree $\tau_1 \cdots \tau_k$, whose root has decoration $m$.
\end{proof}

Following \cite{EbrahimiFardLundervoldMuntheKaas2015}, we now perform the Guin--Oudom extension to construct the Hopf algebra $(\mathfrak{T},\ast,\Delta_{\shuffle})$ for recentering in planar regularity structures:
\begin{align*}
	1 \dgraftt \omega 
	&= \omega, \\
	\tau_1\omega \dgraftt \tau_2 
	&= \tau_1 \dgraftt (\omega \dgraftt \tau_2) - (\tau_1 \dgraftt \omega) \dgraftt \tau_2,\\
	\omega_1 \dgraftt (\omega_2\omega_3) 
	&= ((\omega_1)_{(1)} \dgraftt \omega_2)((\omega_1)_{(2)} \dgraftt \omega_3),
\end{align*}
where $\tau_1,\tau_2 \in \mathcal{V}$ and $\omega,\omega_1,\omega_2,\omega_3 \in \mathfrak{T}$, and:
\begin{align*}
	\omega_1 \ast \omega_2
	&=(\omega_1)_{(1)}\odot ((\omega_1)_{(2)} \dgraftt \omega_2).
\end{align*}

The Grossman--Larson-type product $\ast$ can be interpreted as a deformed grafting product, where we also graft onto the root. Let $(\mathfrak{T},\shuffle,\Delta_{\scriptscriptstyle{\text{DMKW}}})$ be the graded dual Hopf algebra, this is the Hopf algebra for recentering in planar regularity structures. The coproduct $\Delta_{\scriptscriptstyle{\text{DMKW}}}$ is called the deformed Munthe--Kaas--Wright coproduct. A geometric embedding for planar regularity structures would be an isomorphism between $(\mathfrak{T},\ast,\Delta_{\shuffle})$ and a tensor Hopf algebra, we do not achieve such an isomorphism due to the commutativity in the root decorations. We will instead have to follow the non-planar case and quotient by the identities induced by this commutativity.

\begin{remark}
If there were to exist a natural growth operator on $(\mathfrak{T},\shuffle,\Delta_{\scriptscriptstyle{\text{DMKW}}})$, then $(\mathfrak{T},\Delta_{\scriptscriptstyle{\text{DMKW}}})$ would be cofree and $(\mathfrak{T},\ast)$ would be free. However, consider the single-vertex trees, for $i \neq j$:
\begin{align*}
	X^i \ast X^j = X^{i+j}=X^j \ast X^i.
\end{align*}
Because there exist distinct elements that commute, the product is not free associative. Hence there can not exist a natural growth operator, and the Hopf algebra can not be isomorphic to a tensor algebra. This same observation can be made also in the non-planar setting considered in \cite{BrunedKatsetsiadis2023}.
\end{remark}

Consider again the space $\mathcal{V}$, and the tensor algebra $T(\mathcal{V})$. Furthermore consider the ideal $\mathfrak{J}$ generated by $a \otimes b - b \otimes a = [a,b]_0$. It is then a standard result that $(T(\mathcal{V})/\mathfrak{J},\otimes,\Delta_{\shuffle}   )$ is isomorphic to the universal enveloping Hopf algebra $(\mathfrak{T},\odot,\Delta_{\shuffle}  )$, where now we use $\odot$ for the tensor product on the quotient space. Note that the ideal $\mathfrak{J}$ corresponds exactly to the quotient taken in \cite{BrunedKatsetsiadis2023} for the non-planar setting. It remains to show a Hopf algebra isomorphism between $(\mathfrak{T},\ast,\Delta_{\shuffle})$ and  $(\mathfrak{T},\odot,\Delta_{\shuffle})$. Consider the space:
\begin{align*}
	\mathfrak{B}
	= \text{span}(\{X^i : i=(0,\ldots,0,1,0,\ldots,0)  \}\cup \{I_a(\tau): \tau\in \mathfrak{T}  \}   )
\end{align*}
spanned the units in $\mathbb{N}^d$ and by trees that each of which has exactly one edge adjacent to an undecorated root.

\begin{lemma}
The space $\mathfrak{B}$ generates $(\mathfrak{T},\ast)$.
\end{lemma}

\begin{proof}
We first note that $\mathfrak{B}$ generates all trees with undecorated root, which can be seen by an induction over the number of edge-adjacent roots. Note that
\begin{align*}
	I_a(\omega_1) \ast I_b(\omega_2)
	=I_a( (\omega_1)_{(1)} )\odot I_b( (\omega_1)_{(2)} \dgraftt \omega_2 ),
\end{align*}
which is $I_a(\omega_1)\odot I_b(\omega_2)$ plus terms with a low number of edge-adjacent roots. Next we note that we can generate any root decoration on these trees, again by induction:
\begin{align*}
	X^m \ast I_a(\omega)
	=\sum_{m=n_1+n_2} X^{n_1}I_a(\uparrow^{n_2}\omega  ),
\end{align*}
which is the tree with $m$ as root decoration plus trees with lower root decorations.
\end{proof}

\begin{lemma}
The space $\mathfrak{B}$ generates $(\mathfrak{T},\odot)$.
\end{lemma}

\begin{proof}
It is clear that $\mathfrak{B}$ generates $(\mathcal{V},[\cdot,\cdot]_0)$, hence it must also generate $(\mathfrak{T},\odot)$.
\end{proof}

We can now define a map $\phi: \mathfrak{T} \to \mathfrak{T}$ to be the identity on $\mathfrak{B}$ and extending as an algebra morphism $\phi(\omega_1 \ast \omega_2)=\phi(\omega_1) \odot \phi(\omega_2)$.

\begin{proposition}
The map $\phi$ is a Hopf algebra isomorphism.
\end{proposition}

\begin{proof}
We need to check that $\phi$ is injective and that $\phi$ is a coalgebra morphism. The proof for Proposition \ref{Prop:HopfAlgebraMorphism} applies for these properties.
\end{proof}

%%%%%%%%%%%%%%%%%%%%%%%%%%%%%%%%%%%
%%%%%%%%%%%%%%%%%%%%%%%%%%%%%%%%%%%

\section{Conclusions}
\label{sec:concl}

The Munthe-Kaas--Wright Hopf algebra plays an important role in understanding differential equations on homogeneous spaces, via the notions of LB-series, planarly branched rough paths and regularity structures. These are, in explicit ways, generalizations of the Euclidean space notions of the Butcher--Connes--Kreimer Hopf algebra, B-series, branched rough paths and planar regularity structures. Via the explicit generalizations, many known results from the Euclidean space case can be transfered to their homogeneous space counterparts. The goal of this paper was to review the MKW Hopf algebra, and the transfer of results on the BCK Hopf algebra onto results on the MKW Hopf algebra. In doing so we recalled the Guin--Oudom construction, and how it induces both of the mentioned Hopf algebras. In surveying the Guin--Oudom construction, we discovered a cointeraction of the MKW Hopf algebra onto the shuffle Hopf algebra over the alphabet of planar rooted trees. We then recalled Foissy's results on the BCK Hopf algebra and its natural growth operator, seen from the perspective of any Hopf algebra with a natural growth operator. By then showing that the MKW Hopf algebra has a natural growth operator, we got to apply Foissy's results also in this case. This implies the important result that the MKW Hopf algebra is cofree, and is isomorphic to a shuffle Hopf algebra over a larger alphabet. The Hopf algebra isomorphism has consequences for planarly branched rough paths, which parallels known results for branched rough paths. We showed that any planarly branched rough path can equivalently be expressed as a geometric rough path. We furthermore showed that signatures of planarly branched rough paths are unique up to tree-like equivalence. Finally we took the theme of geometric embeddings to its limits by generalizing Bruned and Katsetsiadis' result on regularity structures. Unlike rough paths, the Hopf algebra for regularity structures is not isomorphic to a shuffle Hopf algebra. Rather, it is isomorphic to a quotient Hopf algebra of a shuffle Hopf algebra. We showed that the Hopf algebra for planar regularity structures is not isomorphic to a shuffle Hopf algebra, and is isomorphic to a quotient Hopf algebra of a shuffle Hopf algebra.

%%%%%%%%%%%%%%%%%%%%%%%%%%%%%%%%%%%
%%%%%%%%%%%%%%%%%%%%%%%%%%%%%%%%%%%
%%%%%%%%%%%%%%%%%%%%%%%%%%%%%%%%%%%
%%%%%%%%%%%%%%%%%%%%%%%%%%%%%%%%%%%

\bibliographystyle{acm}
\bibliography{PrimitiveReferences}
\end{document}